\documentclass[reqno]{amsart}
\usepackage{graphicx,graphics}
\usepackage{subfig}

\newcommand{\Set}[1]{\left(#1\right)}

\newcommand{\Real}{\mathbb{R}}

\newcommand{\set}[1]{\left\{#1\right\}}

\newtheorem{lem}{Lemma}
\newtheorem{thm}{Theorem}

\newtheorem{rem}{Remark}
\newtheorem{defn}{Definition}

\numberwithin{thm}{section}
\numberwithin{lem}{section}
\numberwithin{cor}{section}
\numberwithin{rem}{section}
\numberwithin{defn}{section}
\numberwithin{ex}{section}
\numberwithin{prop}{section}
\numberwithin{equation}{section}

\numberwithin{equation}{section}

\begin{document}

\centerline {\textsc {\large On tail behaviour of $k$-th upper order statistics}}
\centerline {\textsc{\large under fixed and random sample sizes via tail equivalence}} 
\vspace{0.2in}

\begin{center}
   Sreenivasan Ravi\footnote{Research work supported by University Grants Commission Major Research Project Stat-2013-19168; Corresponding author: ravi@statistics.uni-mysore.ac.in, stat.ravi@gmail.com} and
   Mandagere Chandrashekhar Manohar \footnote{Thanks his employer Acharya Institute of Technology, Bengaluru, for support to carry out this work; mcmanoharmc@gmail.com } \\
   Department of Studies in Statistics\\
University of Mysore, Manasagangotri, Mysuru 570006, India
\end{center}

\vspace{0.2in}

\noindent {\bf Abstract:} For a fixed positive integer $\;k,\;$ limit laws of linearly normalized $\;k$-th upper order statistics are well known. In this article, a comprehensive study of tail behaviours of limit laws of normalized $k$-th upper order statistics under fixed and random sample sizes is carried out using tail equivalence which leads to some interesting tail behaviours of the limit laws. These lead to definitive answers about their max domains of attraction. Stochastic ordering properties of the limit laws are also studied. The results obtained are not dependent on linear norming and apply to power norming also and generalize some results already available in the literature. And the proofs given here are elementary.

%\vspace{0.5in}

\vspace{0.1in} \noindent {\bf Keywords \& phrases:} Tail behaviour, tail equivalence, $k$-th extremes, partial maxima, max domains of attraction, stochastic ordering, Burr distribution. 

%\vspace{0.5in}

\vspace{0.1in} \noindent {\bf MSC 2010 classification:} 60F10

%#############################################################################
%\newpage

\section{\bf Statement of the problem with motivation}
Extreme value laws are limit laws of linearly normalized partial maxima $M_n=$ \\ $\max\{X_1, \ldots, X_n\}$ of independent and identically distributed (iid) random variables (rvs) $X_1,X_2,\ldots,$ with common distribution function (df) $F,$ namely,
\begin{equation}\label{Introduction_e1}
	\lim_{n\to\infty}P(M_n\leq a_nx+b_n)=\lim_{n\to\infty}F^n(a_nx+b_n)=G(x),\;\;x\in \mathcal{C}(G),
\end{equation}
	where $a_n>0,$ $b_n\in\Real,$ are norming constants, $G$ is a non-degenerate df, $\mathcal{C}(G)$ is the set of all continuity points of $G.$ If, for some non-degenerate df $G,$ a df $F$ satisfies (\ref{Introduction_e1}) for some norming constants $a_n>0,$ $b_n\in\Real,$ then the df $F$ is said to belong to the max domain of attraction of $G$ under linear normalization and we denote it by $F\in \mathcal{D}_l(G).$ Extreme value laws $G$ satisfying (\ref{Introduction_e1}) are types of the following distributions, namely, the Fr\'{e}chet law, $\; \Phi_\alpha(x) = \exp(-x^{-\alpha}),$ $0\leq x;\;$ the Weibull law, $\; \Psi_\alpha(x) =$ $\exp(- |x|^{\alpha}),$ $ x<0;\;$ and the Gumbel law, $\; \Lambda(x)$ $ = \exp(-\exp(-x)),$ $ x\in\Real;$ $\alpha>0\;$ being a parameter. Here and elsewhere in this article dfs are specified only for $\;x\;$ values for which they belong to $\;(0,1)\;$ and probability density functions (pdfs) are specified for values wherein they are positive. Two dfs $G$ and $H$ are said to be of the same type if $G(x) = H(Ax + B), x \in \Real,$ for some constants $A > 0, B \in \Real.$ Extreme value laws $G$ satisfy the stability relation $G^n(A_n x + B_n) = G(x), x \in \Real, n \geq 1,$ where $A_n > 0, B_n \in \Real$ are constants as given below and hence they are also called as max stable laws. The stability relation means that df of linearly normalized maxima of a random sample of size $n$ from df $G$ is equal to $G$ for every $n.$ We have, for $n \geq 1$ and $x \in \Real,$ 
\begin{equation} \label{stable}
			\Phi_\alpha^n\left(n^{1/ \alpha}x\right)  = \Phi_\alpha(x); \;\; 				    \Psi_\alpha^n\left(n^{-1 /\alpha} x\right) =  \Psi_\alpha(x); \;\; 	\text{and}\;\; \Lambda^n(x + \log n) = \Lambda(x). \end{equation}
Note that (\ref{Introduction_e1}) is equivalent to
\begin{eqnarray}\label{Introduction_e1.2}
\lim_{n\to\infty}n\{1-F(a_nx+b_n)\}=-\log G(x), \; x \in \{y\in \Real: G(y) > 0\}.
\end{eqnarray}
Criteria for $F\in \mathcal{D}_l(G)$ are well known (see, for example, Galambos, 1987; Resnick, 1987; Embrechts et al., 1997). These are repeated in the Appendix for ease of reference. 

Let $X_{1:n} \leq X_{2:n} \leq \ldots \leq X_{n:n}$ denote the order statistics from the random sample $\;\{X_1, \ldots, X_n\} \;$ from the df $\;F\;$ and $F \in \mathcal{D}_l(G)$ for some non-degenerate df $G$ so that $F$ satisfies (\ref{Introduction_e1}) for some norming constants $a_n > 0, b_n \in \Real.$ It is well known (see, for example, Galambos, 1987) that the df of the $k$-th upper order statistic $\;X_{n-k+1:n},\;$ for a fixed positive integer $k,$ is given by $\;F_{k:n}(x) = P\left( X_{n-k+1:n} \leq x \right) = \sum_{i=0}^{k-1} \binom ni  F^{n-i}(x) (1 - F(x))^{i}, x \in \Real,$ and the limit $G_k(x) = \lim_{n\to\infty} F_{k:n}(a_n x + b_n) $ is given by 
\begin{equation}\label{eqn:r_max_df}
G_k(x) = G(x) \sum_{i=0}^{k-1} \frac{(- \log G(x))^i}{i!},\;\;  x \in \{y: G(y) > 0\}. 
\end{equation} 

An obvious question is whether $G_k$ satisfies some stability relation like (\ref{stable}). Though we are unable to give a satisfactory answer to this question, we discuss implication of such a stability condition in this article. We show that $\;G_k\;$ is tail equivalent to $\;1 - (1-F(\cdot))^k\;$ and look at the consequences of this, one of which, for example, is that $\;G_k \in \mathcal{D}_l(\Phi_{k\alpha})\;$ if $F\in \mathcal{D}_l(\Phi_{\alpha}).\;$ These lead to generation of an hierarchy of Fr\'{e}chet laws with different integer exponents.   Similar questions are also discussed when the sample size $\;n\;$ is a rv. A few of the results obtained here are available in Peng et al. (2010) and Barakat and Nigm (2002) but the proofs given here are elementary using tail equivalence unlike the proofs in these two articles. Results given here also generalize results given in Peng et al. (2010) and hold under power norming also. Analogous results for lower order statistics can be written down using results proved in this article. 

This article is organized as follows. The next section discusses tail equivalence of $\;G_k \;$ and its consequences when sample size $\;n\;$ is fixed. The limiting distribution of normalized $\;k$-th upper order statistic $\;X_{n-k+1:n}\;$ when the sample size $\;n\;$ is discrete uniform rv is discussed in Section 3 and in Section 4 when the random sample size is binomial, Poisson, logarithmic, geometric, and negative binomial. Tail behaviour of the limit law obtained in Barakat and Nigm (2002) is also discussed in this section. Discrete uniform random maxima appears for the first time in the literature in this article though this was announced by the first author in the International Congress of Mathematicians 2006 held at Madrid, Spain (see Ravi, 2006). Some examples are given in Section 5 and stochastic orderings are discussed in Section 6. Proofs of some of the results are given after the statements and proofs of some main results along with some preliminary results are given in Section 7. Section 8 has some known results used in the article which are repeated here for ease of reading. 

Throughout this article we use the following notations. For a df $\;F,\;$ its left extremity is denoted by $\;l(F)= \inf \set{x: F(x)>0},\;$ its right extremity by $\;r(F) = \sup\{x \in \Real: F(x) < 1\},$ and $\;F^{-}(y) = \inf \left\{ x : F(x) \geq y \right\}, y \in \Real. \;$
The symbol  $\; \stackrel{d}\rightarrow \;$ denotes convergence in distribution,  for two functions $\;g(.)\;$ and $\;h(.),\;$ $g(x) \sim h(x)\;$ as $\;x \rightarrow a\;$ if $\;\lim\limits_{x\rightarrow a} \dfrac {g(x)}{h(x)}=1\;$ so that $\;-\ln x \sim 1-x\;$ as $\;x\rightarrow 1\;$ and $\;-\ln F(x) \sim 1-F(x)\;$ as $\;x\rightarrow r(F)\;$ for any df $\;F.\;$ Further,  $\;g(x) \sim o( h(x))\;$ as $\;x \rightarrow a\;$ if $\;\lim\limits_{x\rightarrow a} \dfrac {g(x)}{h(x)}=0.$

\section{\bf Limiting behaviour of $\;k$-th upper order statistic when sample size is fixed}
We study a possible stability relation for the df $\;F_{k:n}\;$ of the $\;k$-th upper order statistic from a random sample of size $\;n\;$ from a df $\;F,\;$ where $\;k\;$ is a fixed positive integer. For convenience, wherever necessary, we assume that the df $\;F\;$ is absolutely continuous with pdf $\;f.$ 

\subsection{On stability relation for df of $\;k$-th upper order statistic}	
\begin{thm} If $\;F\in D_l(G)\;$ for some max stable df $\;G\;$ so that $\;F\;$ satisfies (\ref{Introduction_e1}) for some norming constants $\;a_n>0,b_n \in \Real,\;$ and $\;F_{k:n}\;$ satisfies the stability relation $\;F_{k:n}(a_n x+b_n)=F(x),\;$ $  x\in R,\;$ $n\ge 1,\;$  then $\;F(x) = 0\;$ if $\;G(x) = 0\;$ and 
\begin{eqnarray} \label{Stability-k}
 F(x)=G(x) \sum_{i=0}^{k-1} \frac {(-\ln G(x))^i}{i!}, \;\;\;\;x \in \{y: G(y) > 0\}, \;\;G =\Phi_\alpha\; \mbox{or} \; \Psi_\alpha \; \mbox{or} \; \Lambda. 
\end{eqnarray}
\end{thm}

\begin{proof}[\textbf {Proof:}] We have $\;F_{k:n}(x) = P(X_{n-k+1:n} \leq x)$ $= \sum_{i=0}^{k-1} \binom ni F^{n-i}(x) (1-F(x))^{i}, \; x \in \Real.\;$ If
\begin{equation} F_{k:n}(a_n x+b_n) = \sum_{i=0}^{k-1} \binom ni F^{n-i}(a_nx+b_n) (1-F(a_nx+b_n))^{i}  = F(x), \label{eqn:r}\end{equation}
then taking limit as $\;n \rightarrow \infty\;$ in (\ref{eqn:r})  we have $\;F(x) = \lim_{n \rightarrow \infty} F_{k:n}(a_nx+b_n) $
\begin{eqnarray}
 & = & \lim_{n \rightarrow \infty}  \sum_{i=0}^{k-1} \frac 1{i!} \left(F^n(a_nx+b_n)\right)^{1-\frac in} \prod_{j=0}^{i-1} \left( \frac {n-j}n n(1-F(a_nx+b_n))\right) \nonumber \\
&=&  \sum_{i=0}^{k-1}\frac 1{i!} G(x)   \prod_{j=0}^{i-1} (-\ln G(x)) =G(x) \sum_{i=0}^{k-1}\frac{(-\ln G(x))^i}{i!}   \label{nonrandomlimit}
\end{eqnarray} so that (\ref{Stability-k}) holds.
\end{proof}

\begin{rem} For $\;k=1,$ $F_{1:n}\;$ is the df of the partial maxima $\;X_{n:n}\;$ and in this case if $\;F=G,\;$ then $\;G\;$ satisfies the stability relation (\ref{stable}). However, though $\;G \in D_l(G),\;$ note that, with $\;F = G\;$ in (\ref{eqn:r}), we end up with $\;G(x) \sum_{i=0}^{k-1} \frac {(-\ln G(x))^i}{i!} = G(x),\;$ and hence $\;G\;$ has to be a degenerate df, a contradiction. This means that the $\;k$-th upper order statistic from a random sample from one of the extreme value distributions will not satisfy a property like (\ref{eqn:r}). However, we show later that if a df $\;F \in D_l(G) \;$ for some $\;G,\;$ then a function of the type $\;F(x) \sum_{i=0}^{k-1} \dfrac {(-\ln F(x))^i}{i!} \;$ is a df and belongs to $\;D_l(G)\;$ with possibly a different exponent.  \label{rem:r}
\end{rem}

\begin{defn} \label{Def-Fk} For any df $\;F\;$ and fixed integer $\;k \ge 1,\;$ define $F_k(x)=F(x)\sum\limits_{i=0}^{k-1} \dfrac {(-\ln F(x))^i}{i!},\;$ $x \in \{y: F(y)>0\}.\;$ \end{defn}

\begin{thm}
Let $\;X\;$ have absolutely continuous df $\;F \;$ with pdf $\;f.\;$ Then the df $\;F_k \;$ satisfies the recurrence relation 
 \begin{equation} F_{k+1}(x)=F_k(x)+\dfrac {F(x)}{k!} (-\ln F(x))^k, \quad k\ge 1, \;\;x\in \{y: F(y)>0\}. \label{eqn:recdf} \end{equation}
Further, the pdf of $\;F_{k+1}\;$ is given by  
\begin{equation} f_{k+1}(x)=\dfrac{f(x)}{k!} (-\ln F(x))^k, \quad k\ge 1, \;\;x\in \{y: F(y)>0\}.\label{eqn:recpdf} \end{equation}  \label{thm:r-max-df}
\end{thm}

\begin{proof}[\textbf{Proof:}] The proof of the relation (\ref{eqn:recdf}) is evident from the definition of $F_k\;$ since 
\[ F_{k+1}(x)-F_k(x)=F(x)\sum_{i=0}^k \dfrac{(-\ln F(x))^i}{i!}-F(x)\sum_{i=0}^{k-1}\dfrac{(-\ln F(x))^i}{i!}=\dfrac{F(x)}{k!}(-\ln F(x))^k.\]
We use induction to prove (\ref{eqn:recpdf}). 
For $k=1$, we have $\;F_2(x)=F(x)(1-\ln F(x)).\;$ By differentiation, $\;f_2(x)=$ $f(x)(1-\ln F(x))$ $ + F(x) \left(-\dfrac {f(x)}{F(x)}\right)$ $=f(x)(-\ln F(x))\ge 0,\;$ which is (\ref{eqn:recpdf}) for $\;k=1.\;$ Assuming that (\ref{eqn:recpdf}) is true for $\;k=m,\;$ we have $ F_{m+2} (x)=F_{m+1}(x)+\dfrac {F(x)}{(m+1)!} (-\ln F(x))^{m+1},$ and upon differentiating with respect to $\;x,\;$ we get $\;f_{m+2} (x) = $ $ f_{m+1}(x)+\frac{f(x)}{(m+1)!} \set{-(m+1)(-\ln F(x))^m + (-\ln F(x))^{m+1}}$ $ = \dfrac {f(x)}{(m+1)!} (-\ln F(x))^{m+1},\;$ which is (\ref{eqn:recpdf}) with $\;k = m+1.\;$ Finally since $\;f(x)\ge 0\;$ and $\;0<F(x)<1,\;$ it follows that $\;f_{k+1}(x)\ge 0,\;$ completing the proof.
\end{proof}

The following theorem whose proof is given in Section 7 shows that any positive power of tail of a df is also a tail and looks at the tail behaviour of the tail $\;(1-F(.))^k \;$ and this will be used to study the tail behaviour of $\;F_k \;$ in the next result. 

\begin{thm} If $\;F\;$ is a df with pdf $\;f,\;$  then for every positive integer $\;k,\;$ $\;\left(1-F(x)\right)^k\;$ is tail of the df $\;H_k(x)=1-\left(1-F(x)\right)^k,\ x\in \Real, \;$ and $\;H_k\;$ is also absolutely continuous with pdf $\; H_k^\prime(x)=k\set{1-F(x)}^{k-1} f(x), x \in \Real.$ Further, the following are true. 
\begin{enumerate}
\item[(a)] If $\;F\in D_l(\Phi_\alpha),\;$ then $\;r(H_k)=r(F)=\infty,\;$ and $\;H_k \in D_l(\Phi_{k\alpha})\;$ so that (\ref{Introduction_e1}) holds with $\;F = H_k, G = \Phi_{k\alpha},\;$ $a_n =F^-(1 - (1/n)^{1/k}), b_n=0.$ 
\item[(b)] If $\;F\in D_l(\Psi_\alpha)\;$ then $\;r(H_k)=r(F)<\infty,\;$  and $\;H_k \in D_l(\Psi_{k\alpha})\;$ so that (\ref{Introduction_e1}) holds with $\;F = H_k, G = \Psi_{k\alpha},\;$ $a_n =r(F)-F^- (1 -  (1/n)^{1/k}), b_n=r(F).$
\item[(c)] If $\;F\in D_l(\Lambda),$ $\;a_n = v(b_n)\;$ and
$\;b_n = F^{-}\left(1 - \frac{1}{n}\right) \;$ as in (c) of Theorem $\ref{T-l-max}$ so that (\ref{Introduction_e1}) holds with $\;G = \Lambda,\;$ then $\;r(H_k)=r(F),\;$ and $\;H_k  \in D_l(\Lambda)\;$ so that (\ref{Introduction_e1}) holds with $\;F = H_k, G = \Lambda,\;$  $\;a_n = \dfrac{v(b_n)}{k}, b_n = H_k^-( 1 - 1/n).\;$ 
\end{enumerate} \label{lem:1} \end{thm}

\begin{rem}If $\;F\;$ is a df, for every positive integer $\;k,$ $F^k\;$ is a df and  if $\;1-F\;$ is  tail of a df then $\;(1-F)^k\;$ is also a tail.  \label{rem:1}\end{rem}

The next result is for fixed sample size and its proof is given in Section 7. 

\begin{thm} Let rv $\;X\;$ have absolutely continuous df $\;F\;$ with pdf $\;f\;$ and $\;k\;$ be a positive integer. Then for $\; F_k(x)=F(x)\sum_{i=0}^{k-1} \dfrac {(-\ln F(x))^i}{i!},$ $\;x \in \{y: F(y) > 0\},\;$ the following results are true.
\begin{enumerate}
\item[(a)] $F_k\;$ is a df with $\;r(F_k) = r(F),\;$ pdf $\;f_k(x)=\dfrac{f(x)}{(k-1)!} (-\ln F(x))^{k-1},\;$ $x \in \{y \in \Real: F(y) > 0\};$ and $\;\lim_{x \rightarrow r(F)} \dfrac{1-F_k(x)}{(1-F(x))^k} = \dfrac{1}{k!},\;$ so that $\;F_k\;$ is tail equivalent to $\;H_k.$ 
\item[(b)]  If $\;F\in D_l(\Phi_\alpha),\;$ then $\;r(F_k)=r(F)=\infty,\;$ and $\;F_k \in D_l(\Phi_{k\alpha})\;$ so that (\ref{Introduction_e1}) holds with $\;F = F_k, G = \Phi_{k\alpha},\;$ $ \;a_n =F^- (1 - (k!/n)^{1/k}), b_n=0.\;$ 
\item[(c)] If $\;F\in D_l(\Psi_\alpha)\;$ then $\;r(F_k)=r(F)<\infty,\;$ and $\;F_k \in D_l(\Psi_{k\alpha})\;$ so that (\ref{Introduction_e1}) holds with $\;F = F_k, G = \Psi_{k\alpha},\;$ $\;a_n =r(F) - F^-(1 - (k!/n)^{1/k}), b_n=r(F).\;$ 
\item[(d)] If $\;F\in D_l(\Lambda),$ $\;a_n = v(b_n)\;$ and
$\;b_n = F^{-}\left(1 - \frac{1}{n}\right) \;$ as in (c) of Theorem $\ref{T-l-max}$ so that (\ref{Introduction_e1}) holds with $\;G = \Lambda,\;$ then $\;r(F_k)=r(F),\;$ and $\;F_k  \in D_l(\Lambda)\;$ so that (\ref{Introduction_e1}) holds with $\;F = F_k, G = \Lambda,\;$ $\;a_n = \dfrac{v(b_n)}{k}, b_n = F_k^-( 1 - 1/n).\;$ 
\end{enumerate} \label{Thm-Fk}
\end{thm}

\section{\bf Limiting behaviour of $k$th upper order statistics when sample size is discrete uniform rv}

In this section, we assume that the sample size $\;n\;$ in the previous section is a discrete uniform rv $\;N_n\;$ with probability mass function (pmf) $\;P(N_n=r)=\frac 1n,$ $ r=m+1,m+2,\ldots,m+n,\;$ $N_n\;$ independent of the iid rvs $\;X_1, X_2, \ldots,\;$ where $\;m\ge 1\;$ is a fixed integer. The tail behaviour of the limit of linearly normalized $\;k$-th upper order statistic $\;X_{N_n-k+1:N_n}\;$ which is the $\;k$-th maximum among $\; \set{X_1,X_2,\ldots, X_{N_n}}\;$ is discussed here. Observe that $\;X_{N_n-k+1:N_n}$ is well defined for $\;1\le k \le m.\;$ 

For fixed integer $\;k, \, 1 \le k \le m,\;$ the df of the $\;k$-th upper order statistic of a sequence of iid rvs with random sample size $\;N_n\;$ is given by 
\begin{eqnarray}
F_{k:N_n}(x)&=&P(X_{N_n-k+1:N_n} \le x ) = \sum_{r=m}^\infty  P(X_{N_n-k+1:N_n} \le x , N_n=r) \nonumber 
\\&=& \sum_{r=m}^\infty \sum_{i=0}^{k-1} \binom ri  F^{r-i}(x) (1-F(x))^i P(N_n=r), \; x \in \Real. \label{eqn:gr-max1}
\end{eqnarray}

\begin{thm} If $\;F\in D_l(G)\;$ for some max stable df $\;G\;$ so that $\;F\;$ satisfies (\ref{Introduction_e1}) for some norming constants $\;a_n>0,b_n \in \Real,\;$
then for fixed integer $\;k, 1 \le k \le m,\;$ the limiting distribution $\;\lim_{n \rightarrow \infty} F_{k:N_n}(a_n x + b_n)\;$ is equal to   
\begin{equation} J_k(x) =  k\Set{\frac {1-G(x)}{-\ln G(x)}}-G(x) \sum_{l=1}^{k-1} (k-l) \frac {(-\ln G(x))^{l-1}}{l!}, \; x \in \{y \in \Real: G(y) > 0\},  \label{eqn:dr-max} \end{equation} where  $\; G =\Phi_\alpha \;\mbox{or}\; \Psi_\alpha \;\mbox{or}\;  \Lambda.$ \label{Thm_DU1}
\end{thm}

\begin{defn} \label{Def-Fk} For any df $\;F,\;$ and fixed integer $\;k \ge 1,\;$ we define $U_k(x)=k\Set{\frac {1-F(x)}{-\ln F(x)}}-F(x) \sum_{l=1}^{k-1} (k-l) \frac {(-\ln F(x))^{l-1}}{l!},\;$ $x \in \{y: F(y)>0\}.\;$ \end{defn}

\begin{thm} If rv $\;X\;$ has absolutely continuous df $\;F\;$ with pdf $\;f,\;$ $\;k\;$ is a fixed positive integer, and $\;U_1(x)=\dfrac {1-F(x)}{-\ln F(x)},\;$ $x \in \{y: F(y) > 0\},\;$ then $\;U_1\;$ is a df with $\;r(U_1) = r(F),\;$ pdf $\;u_1(x)=\dfrac {f(x)}{F(x)}\dfrac {U_1(x)-F(x)}{-\ln F(x)}=\dfrac {f(x)}{F(x)}\Set{\dfrac {1-F(x)+F(x) \ln F(x)}{(-\ln F(x))^2}},\; x \in \{y \in \Real: F(y) > 0\};\;$ and $\;\lim_{x \rightarrow r(F)} \dfrac{1-U_1(x)}{1-F(x)} = \dfrac 12\;$ so that $\;U_1\;$ is tail equivalent to $\;F.$ 
 \label{Thm-FkW}
\end{thm}

\begin{rem} $U_1\;$ in the above theorem was derived first and shown to be a df in Ravi (2006) and this and the next result generalize the results obtained in Ravi (2006). 
\end{rem}

\begin{thm} Let rv $\;X\;$ have absolutely continuous df $\;F\;$ with pdf $\;f\;$ and $\;k\;$ be a fixed positive integer. Then for $\;U_k\;$ as in Definition \ref{Def-Fk} and $\;H_k(x) = 1 - (1-F(x))^k, x \in \Real,\;$ the following results are true.
\begin{enumerate}
\item[(a)] $U_k\;$ is a df with $\;r(U_k) = r(F),\;$ pdf $\;u_k(x)=\dfrac {kf(x)}{(-\ln F(x))^2}\Set{\dfrac 1{F(x)} - \sum_{l=0}^k  \dfrac {(-\ln F(x))^l}{l!}},\;$ $x \in \{y \in \Real: F(y) > 0\}.$
\item[(b)] $\;\lim_{x \rightarrow r(F)} \dfrac{1-U_k(x)}{(1-F(x))^k} = \dfrac{1}{(k+1)!},\;$ so that $\;U_k\;$ is tail equivalent to $\;H_k.$ 
\item[(c)]  If $\;F\in D_l(\Phi_\alpha),\;$ then $\;r(U_k)=r(F)=\infty,\;$ $\;U_k \in D_l(\Phi_{k\alpha})\;$ so that (\ref{Introduction_e1}) holds with $\;F = U_k, G = \Phi_{k\alpha},\;$ $\;a_n =F^- (1 - ((k+1)!/n)^{1/k}), b_n=0.$
\item[(d)] If $\;F\in D_l(\Psi_\alpha)\;$ then $\;r(U_k)=r(F)<\infty,\;$  $\;U_k \in D_l(\Psi_{k\alpha})\;$ so that (\ref{Introduction_e1}) holds with $\;F = U_k, G = \Psi_{k\alpha},\;$ $\;a_n =r(F) - F^-(1 - ((k+1)!/n)^{1/k}), b_n=r(F).$
\item[(e)] If $\;F\in D_l(\Lambda),$ $\;a_n = v(b_n)\;$ and
$\;b_n = F^{-}\left(1 - \frac{1}{n}\right) \;$ as in (c) of Theorem $\ref{T-l-max}$ so that (\ref{Introduction_e1}) holds with $G = \Lambda,$ then $\;r(U_k)=r(F),\;$ $\;U_k  \in D_l(\Lambda)\;$ so that (\ref{Introduction_e1}) holds with $\;F = U_k, G = \Lambda,\;$ $\;a_n = \dfrac{v(b_n)}{k}, b_n = U_k^-( 1 - 1/n).$ 
\end{enumerate} \label{Thm-FkDU}
\end{thm}

\begin{rem} Note that df $\;U_k\;$ satisfies the recurrence relation $\displaystyle U_{k+1}(x)= U_k(x)+U_1(x)-F(x) \sum_{l=1}^k \frac {(-\ln F(x))^{l-1}}{l!}.$ For, $\;U_{k+1}(x)-U_k(x)=$ 
\begin{eqnarray*}
&=& (k+1)U_1(x)-F(x) \sum_{l=1}^k (k+1-l) \frac {(-\ln F(x))^{l-1}}{l!} - kU_1(x)+F(x) \sum_{l=1}^{k-1} (k-l) \frac {(-\ln F(x))^{l-1}}{l!}
\\&=& U_1(x) - F(x) \sum_{l=1}^{k-1} (k+1-l-k+l) \frac {(-\ln F(x))^{l-1}}{l!} -F(x) \frac {(-\ln F(x))^{k-1}}{k!}
\\&=& U_1(x) - F(x) \sum_{l=1}^k \frac {(-\ln F(x))^{l-1}}{l!}.
\end{eqnarray*}

\end{rem}

\section{\bf Limiting behaviour of $k$th upper order statistic when sample size is random}

In the first subsection we consider the cases when the sample size $\;N_n\;$ follows shifted binomial, Poisson, and logarithmic distributions and show that the limit distribution in all these cases is the same as that in the fixed sample size case. 
After this we consider the cases when the sample size follows shifted geometric and negative binomial and derive non-trivial results.  
\subsection{Binomial, Poisson and Logarithmic $\;k$-th upper order statistic}
\begin{thm} \label{Thm_FkB} 
\begin{itemize}
\item[(a)] Assume that $\;N_n\;$ is a shifted binomial rv with  $\;P(N_n=r)=\binom n{r-m}p_n^{r-m}q_n^{m+n-r},$ $ r=m,m+1,m+2,\ldots,m+n,\;$ for some integer $\;m \geq 1,\;$ with $\;\lim_{n \rightarrow \infty} p_n = 1. \;$ If $\;F\in D_l(G)\;$ for some max stable df $\;G\;$ so that $\;F\;$ satisfies (\ref{Introduction_e1}) for some norming constants $\;a_n>0,b_n \in \Real,\;$
then for fixed integer $\;k, 1 \le k \le m,\;$ the limiting distribution $\;\lim_{n \rightarrow \infty} F_{k:N_n}(a_n x + b_n)\;$ is same as in (\ref{nonrandomlimit}). 
\item[(b)] Assume that $\;N_n\;$ is a shifted Poisson rv with  $\;P(N_n=r)=\dfrac {e^{-\lambda_n} \lambda_n^{r-m}}{(r-m)!},$ $ r=m,m+1,m+2,\ldots,\;$ for some integer $\;m \geq 1,\;$ with $\;\lim_{n \rightarrow \infty} \dfrac{\lambda_n}{n} = 1.$ 
If $\;F\in D_l(G)\;$ for some max stable df $\;G\;$ so that $\;F\;$ satisfies (\ref{Introduction_e1}) for some norming constants $\;a_n>0,b_n \in \Real,\;$
then for fixed integer $\;k, 1 \le k \le m,\;$ the limiting distribution $\;\lim_{n \rightarrow \infty} F_{k:N_n}(a_n x + b_n)\;$ is same as in (\ref{nonrandomlimit}). 
\item[(c)] Assume that $\;N_n\;$ is a shifted logarithmic  rv with  $\;P(N_n=r)=\dfrac1{-\ln (1-\theta_n)} \dfrac{\theta_n^{r-n}}{r-n},$ $ r=n+1,n+2,\ldots,\;$ for some $\;\theta_n, 0 < \theta_n < 1,\;$ with $\;\lim_{n \rightarrow \infty} \dfrac{\theta_n}{-\ln (1-\theta_n)} = 1.$ If $\;F\in D_l(G)\;$ for some max stable df $\;G\;$ so that $\;F\;$ satisfies (\ref{Introduction_e1}) for some norming constants $\;a_n>0,b_n \in \Real,\;$
then for fixed integer $\;k, \,k \ge 1,\;$ the limiting distribution $\;\lim_{n \rightarrow \infty} F_{k:N_n}(a_n x + b_n)\;$ is same as in (\ref{nonrandomlimit}). 
\end{itemize}
\end{thm}

\subsection{Geometric $\;k$-th upper order statistic}
Assume that $\;N_n\;$ is a shifted geometric rv with pmf $\;P(N_n=r)=p_nq_n^{r-m}, r=m,m+1,m+2,\ldots,\;$ where $\;0< p_n < 1, q_n = 1-p_n\;$ and $\;\lim_{n \rightarrow \infty} np_n = 1.$ 
\begin{thm} If $\;F \in D_l(G)$ for some max stable law $\;G,\;$ then for fixed integer $\;k, 1 \le k \le m,\;$ the limiting distribution $\;\lim_{n \rightarrow \infty} F_{k:N_n}(a_n x + b_n)\;$ is equal to  
\begin{equation} L_k(x)=1-\left(\frac{-\ln G(x)}{1-\ln G(x)}\right)^k, \; x \in \{y \in \Real: G(y) > 0\}. \label{eqn:gr-max} \end{equation}  \label{Thm_FkG}
\end{thm}

\begin{rem} Note that $\;L_k(x)= \begin{cases} 1-\left(\frac 1{1+x^\alpha}\right)^k & \;\text{if } G(x)=\Phi_\alpha(x), \\ 1-\left(\dfrac {(-x)^\alpha}{1+(-x)^\alpha}\right)^k & \; \text{if } G(x)=\Psi_\alpha(x), \;\;\;\; \\ 1-\left(\frac {e^{-x}}{1+e^{-x}}\right)^k & \;\text{if } G(x)=\Lambda(x), \end{cases}\;$ so that the first two are Burr distributions of XII kind which is discussed later, and the last is the logistic distribution. 
\end{rem}
Limit distributions of extremes of a random number of random variables has been obtained for three extreme value distributions $\Phi_\alpha,\Psi_\alpha,\Lambda$ in  Vasantalakshmi, M.S.,(2010) using result of Barakat et al(2002). The Burr-family is also introduced in the same.  In the thesis, only case when $N_n$ follow geometric law was considered. In this paper an attempt is made to study the tail behaviour of $k-$th upper order statistic when the df $F\in D_l(G)$ and the results is considered for different distributions of $N_n$. \\ 

As we did in Definition \ref{Def-Fk}, we now define a new function and show that it is a df and study its properties similar to those in Theorem \ref{Thm-Fk}.

\begin{defn}  For any df $\;F,\;$ and fixed integer $\;k \ge 1,\;$ we define $\;R_k(x)=1-\left( \dfrac {-\ln F(x)}{1-\ln F(x)}\right)^k,\;$ $x \in \{y: F(y)>0\}.\;$ \label{Def-Fk2} \end{defn}

\begin{thm} If rv $\;X\;$ has absolutely continuous df $\;F\;$ with pdf $\;f\;$ $\;k\;$ is a positive integer and $\; R_k\;$ is as defined above, then the following results are true.
\begin{enumerate}
\item[(a)] $R_k\;$ is a df with pdf $\;r_k(x)=\dfrac{kf(x) (-\ln F(x))^{k-1}}{F(x) (1-\ln F(x))^{k+1}},\;$ $x \in \{y \in \Real: F(y) > 0\},$ $\;r(R_k) = r(F),\;$ and $\lim_{x \rightarrow r(F)} \dfrac{1-R_k(x)}{(1-F(x))^k} = 1,\;$ so that $\;R_k\;$ is tail equivalent to $\;H_k.$ 
\item[(b)]  If $\;F\in D_l(\Phi_\alpha),\;$ then $\;r(R_k)=r(F)=\infty,\;$  and $\;R_k \in D_l(\Phi_{k\alpha})\;$ so that (\ref{Introduction_e1}) holds with $\;F = R_k, G = \Phi_{k\alpha},$  $ \;a_n =F^- (1 - (1/n)^{1/k}), b_n=0.$ 
\item[(c)] If $\;F\in D_l(\Psi_\alpha)\;$ then $\;r(R_k)=r(F)<\infty,\;$   and 
$\;R_k \in D_l(\Psi_{k\alpha})\;$ so that (\ref{Introduction_e1}) holds with $\;F = R_k, G = \Psi_{k\alpha},\;$ $a_n =r(F) - F^-(1 - (1/n)^{1/k}), b_n=r(F).$ 
\item[(d)] If $\;F\in D_l(\Lambda),\;$ $a_n = v(b_n)\;$ and
$\;b_n = F^{-}\left(1 - \dfrac{1}{n}\right) \;$ as in (c) of Theorem $\ref{T-l-max}$ so that (\ref{Introduction_e1}) holds with $G = \Lambda,$ then $\;r(R_k)=r(F),\;$ and  
$\;R_k  \in D_l(\Lambda)\;$ so that (\ref{Introduction_e1}) holds with $\;F = R_k, G = \Lambda,\;$ $a_n = \dfrac{v(b_n)}{k}, b_n = R_k^-( 1 - 1/n).$ 
\end{enumerate} \label{Thm-FkG2}
\end{thm}

\begin{rem} Burr (1942) proposed twelve explicit forms of dfs which have since come to be known as the Burr system of distributions. These have been studied quite extensively in the literature. A number of well-known distributions such as the uniform, Rayleigh, logistic, and log-logistic are special cases of Burr dfs. A df $\;W\;$ is  said to belong to the Burr family if it satisfies the differential equation  
\begin{equation} \dfrac {dW(x)}{dx}=W(x) (1-W(x)) h(x,W(x))  \label{eqn:burr} \end{equation}
where $\;h(x,W(x))\;$ is a non-negative function for $\;x\;$ for which the function is increasing. One of the forms of $\;h(x,W(x))\;$ is $\;h(x,W(x))=\dfrac {h_1(x)}{W(x)}\;$ where $\;h_1(x)\;$ is a non-negative function. Then (\ref{eqn:burr}) takes the form $\;\dfrac {dW(x)}{dx}=(1-W(x)) h_1(x).\;$ We now show that $\;R_k\;$ is a member of the Burr family of dfs. 
\end{rem}

\begin{thm} The dfs $R_k$ in Definition \ref{Def-Fk2} belong to the Burr family.  \label{Thm-FkG3} \end{thm}

\subsection{Negative Binomial $\;k$-th upper order statistic}
We assume that $\;N_n\;$ is a shifted negative binomial rv with pmf $\;P(N_n=l)=\binom {l-m+r-1}{l-m} p_n^r q_n^{l-m}, r=m,m+1,m+2,\ldots,\;$ where $\;0< p_n < 1, q_n = 1-p_n\;$ and $\;\lim_{n \rightarrow \infty} np_n = 1.$ 
\begin{thm} If $\;F \in D_l(G)$ for some max stable law $\;G,\;$ then for fixed integer $\;k, 1 \le k \le m,\;$ the limiting distribution $\;\lim_{n \rightarrow \infty} F_{k:N_n}(a_n x + b_n)\;$ is equal to  
\begin{equation} S_k(x)=\sum_{l=0}^{k-1} \binom {l+r-1}l \dfrac 1{(1-\ln G(x))^{r}} \left( \dfrac {-\ln G(x)}{1-\ln G(x)}\right)^l, \; x \in \{y \in \Real: G(y) > 0\}. \label{eqn:nbr-max} \end{equation} \label{Thm-FkNB1}
\end{thm}

As we did in Definition \ref{Def-Fk2}, we now define a new function and show that it is a df and study its properties similar to those in Theorems \ref{Thm-Fk} and \ref{Thm-FkG2}.

\begin{defn}  For any df $\;F,\;$ and fixed integer $\;k \ge 1,\;$ we now define $\;T_k(x)=\sum_{l=0}^{k-1} \binom {l+r-1}l  \dfrac {(-\ln F(x))^l}{(1-\ln F(x))^{l+r}},\;$ $x \in \{y: F(y)>0\}.\;$  \label{Def-Fk3} \end{defn}

The following theorem contains a recurrence relation. 

\begin{thm}
The df $\;T_k\;$ in (\ref{eqn:nbr-max}) satisfies the recurrence relation 
 \begin{equation} T_{k+1}(x)= T_k(x)+  \binom {k+r-1}k \dfrac {(-\ln F(x))^k}{(1-\ln F(x))^{k+r}}, \;\; k\ge 1, x \in \Real. \label{eqn:recdf1} \end{equation}
Further its pdf is given by  
\begin{equation} t_{k+1}(x)= \dfrac 1{B(r,k+1)} \dfrac {f(x)}{F(x)}  \dfrac {(-\ln F(x))^k}{(1-\ln F(x))^{r+k+1}}, \;\; k\ge 1, x\in \Real. \label{eqn:recpdf1} \end{equation}  \label{Thm_FkNB2}
\end{thm}

\begin{thm} Let rv $\;X\;$ have absolutely continuous df $\;F\;$ with pdf $\;f\;$ and $\;k\;$ be a fixed positive integer. Then for $\; T_k\;$ as in Definition \ref{Def-Fk2}, the following results are true.
\begin{enumerate}
\item[(a)] $T_k\;$ is a df with pdf $\;t_k(x)=\dfrac 1{B(r,k)} \dfrac {f(x)}{F(x)}  \dfrac {(-\ln F(x))^{k-1}}{(1-\ln F(x))^{r+k}},\;$ $x \in \{y \in \Real: F(y) > 0\},$ right extremity $\;r(T_k) = r(F),\;$ and 
$\lim_{x \rightarrow r(F)} \dfrac{1-T_k(x)}{(1-F(x))^k} = \dfrac{k}{B(r,k)},\;$ so that $\;T_k\;$ is tail equivalent to $\;H_k.$ 
\item[(b)]  If $\;F\in D_l(\Phi_\alpha),\;$ then $\;r(T_k)=r(F)=\infty,\;$  and
$\;T_k \in D_l(\Phi_{k\alpha})\;$ so that (\ref{Introduction_e1}) holds with $\;F = T_k,\;$ $\;G = \Phi_{k\alpha},\;$ $a_n =F^- (1 - (1/n)^{1/k}), b_n=0.$ 
\item[(c)] If $\;F\in D_l(\Psi_\alpha)\;$ then $\;r(T_k)=r(F)<\infty,\;$   and 
$\;T_k \in D_l(\Psi_{k\alpha})\;$ so that (\ref{Introduction_e1}) holds with $\;F = T_k,\;$ $\;G = \Psi_{k\alpha},\;$  $a_n =r(F) - F^-(1 - (1/n)^{1/k}), b_n=r(F).$ 
\item[(d)] If $\;F\in D_l(\Lambda),$ $\;a_n = v(b_n)\;$ and
$\;b_n = F^{-}\left(1 - \dfrac{1}{n}\right) \;$ as in (c) of Theorem $\ref{T-l-max}$ so that (\ref{Introduction_e1}) holds with $\;F = T_k,\;$ $G = \Lambda,$ then $\;r(T_k)=r(F),\;$ and  
$\;T_k  \in D_l(\Lambda)\;$ so that (\ref{Introduction_e1}) holds with $\;F = T_k,\;$ $\;G = \Lambda,\;$ $a_n = \dfrac{v(b_n)}{k}, b_n = T_k^-( 1 - 1/n).$ 
\end{enumerate} \label{Thm_FkNB3}
\end{thm}

\subsection{A general result on tail behaviour of random $k$-th upper order statistics}
The following result, proved in Barakat and Nigm (2002) for order statistics under power normalization, can be proved just by replacing the power normalization there by linear normalization. Here the tail behaviour of the limit law is studied. 

\begin{thm} If $\;\{X_n, n \geq 1\}\;$ is a sequence of iid rvs with df $\;F,\;$ $F \in D_l(G),\;$ $G = \Phi_\alpha \;$ or $\;\Psi_\alpha\;$ or $\;\Lambda,\;$ $k\;$  is a fixed positive integer and the positive-integer valued rv $\;N_n\;$ is such that $\;\frac {N_n}n\;$ converges in probability to $\;\tau,\;$ a positive valued rv, then 
\[ \lim_{n\to  \infty} P(X_{N_n-k+1:N_n} \leq a_nx+b_n) = \sum_{i=0}^{k-1}\dfrac{(-\log G(x))^i}{i!} \int_{0}^{\infty} y^i G^y(x) dP(\tau \leq y), \;\;x \in \Real. \] 
\end{thm}

\begin{thm} If $\displaystyle B_k(x)= \sum_{i=0}^{k-1}\dfrac{(-\log F(x))^i}{i!} E_{\tau}(\tau^i F^\tau(x))\;$ with $\;F \;$ an absolutely continuous df, then
\begin{enumerate}
\item[(a)]  $B_k\;$ is a df with $\;r(B_k) = r(F),\;$ pdf $\displaystyle b_{k}(x)=\dfrac{(-\log F(x))^{k-1}}{(k-1)!} E_{\tau}(\tau^{k} F^{\tau-1}(x)),\;$ and satisfying the recurrence relation $\displaystyle B_{k+1}(x) = B_k(x) +\dfrac{(-\log F(x))^k}{k!} E_{\tau}(\tau^k F^\tau(x)), x \in \{y \in \Real: F(y) > 0\} ;$ 
\item[(b)]  $B_k$ is tail equivalent to $\;H_k.$
\end{enumerate}\label{Thm_FkBN}
\end{thm}

\section{Examples}
In this section we demonstrate the results through some examples. We  consider the standard Pareto, uniform and standard normal distributions. We also consider exponential distribution as a special case of gamma distribution.  For the purpose of specifying the norming constants, we use the following notations: $\;\delta_k=\dfrac 1{k!},\,$ $\theta_k=\frac 1{kB(r,k)}\;$ and $\gamma_k=\frac 1{k!}E_\tau(\tau^k);\;$ and 
\begin{eqnarray*} \eta_k = \left\{ \begin{array}{cll} \delta_k  & \;\; \text{if} & \;\;V_k = F_k, \\ \delta_{k+1} & \; \; \text{if} & \;\;V_k = U_k, \\ \delta_1 & \;\; \text{if} & \;\; V_k = R_k, \\ \theta_k & \;\; \text{if} & \;\; V_k = T_k,  \\ \gamma_k & \;\; \text{if} & \;\; V_k = B_k.  \end{array} \right. \end{eqnarray*}

\begin{enumerate}
\item[(a)] If $F(x)=\Phi_\alpha(x)=\exp \Set{-x^{-\alpha}}, \quad x>0,$ the Frechet law, then $\;V_k \in D_l( \Phi_{k\alpha})$ with norming constants $a_n=(n \eta_k)^{1\over k \alpha},\;$ and $\;b_n = 0.$

\item[(b)] For standard Pareto df with $1-F(x) \sim cx^{-\alpha}, \quad x>1,$ for some constants $c>0,\alpha>0,$ we have $F \in D_l(\Phi_\alpha)$ with the norming constants $b_n=0$ and $a_n=(cn)^{\frac 1\alpha}$ and with $\;V_k \in D_l( \Phi_{k\alpha})$ with norming constants $a_n=(c^k n\eta_k)^{-1\over k \alpha},\;$ and $\;b_n = 0.$

\item[(c)] If  $F(x)=\Psi_\alpha(x)=\exp \Set{-(-x)^\alpha}, \quad x<0,$ the Weibull law, $\;V_k \in D_l( \Psi_{k\alpha})$ with norming constants $\;a_n=(n\eta_k)^{1\over k \alpha}\;$ and $\;b_n = 0.$

\item[(d)]  If $F(x)=x, \quad 0<x<1\;$ is the standard uniform distribution, then $F \in D_l(\Psi_\alpha)$ with norming constants  $a_n=n^{-1}\;$ and $\;b_n = 1,\;$ and $\;V_k \in D_l( \Psi_{k\alpha})$ with norming constants $a_n=(n\eta_k)^{1\over k \alpha}$ and $b_n = 1.$ 

\item[(e)] If $\;F(x)= \dfrac 1{\sqrt{2\pi}} \int_{-\infty}^x  e^{-x^2/2} dx, x \in \Real, \;$ the standard normal distribution, then using Mills' ratio as in Feller (1971), $1-F(x) \sim \frac 1x F^\prime(x)=\frac 1x \frac 1{\sqrt{2\pi}} \exp \Set{-\frac {x^2}2}$. Further $\;F\;$ is a von-Mises function as in (\ref{von-Mises}) and $\;F \in D_l(\Lambda)\;$ with auxiliary function $\;v(x) = \frac {1-F(x)}{F^\prime(x)}\sim \frac 1x.\;$ Then the norming constants can be chosen as $\;a_n=(2\ln n)^{-1/2}, b_n=\sqrt{2\ln n}-\frac 12 \left(\frac {\ln 4\pi +\ln \ln n}{\sqrt{2\ln n}}\right).\;$ Further, $\;V_k \in D_l(\Lambda)\;$ with auxiliary function $\frac 1k v(x)={1\over kx}\;$ and the norming constants can be chosen as $\;a_n=\Set{\frac 2k (\ln n+\ln \delta_1)}^{-1/2}, b_n=\sqrt{\frac {2 (\ln n+\ln \delta)}k}-\frac {\sqrt k} 2 \Set{\frac {\ln 4\pi +\ln (\ln n+\ln \delta) -\ln k}{\sqrt {2\ln n+\ln \delta}}}.$

\item[(f)] Consider a DF $F(x)=1-\exp \Set{-\frac x{1-x}}\; \mbox{if} \; 0 \le x < 1 \;.$ Then simple computations show that $ \lim_{x \to1} \frac {\set{1-F(x)}F^{\prime\prime}(x)}{\set{F^\prime (x)}^2}=\lim_{x \rightarrow 1}  \frac {\set{1-F(x)}^2(1-2x)}{(1-x)^4} \times \frac {(1-x)^4}{\set{1-F(x)}^2}=-1.$  Further $F^{\prime\prime}(x) < 0$   if for $x>\frac 12.$  Then $F\in D_l(\Lambda) $ with the auxiliary function $\;v(x) = \frac {1-F(x)}{F^\prime(x)}=(1-x)^2\;$ and the norming constants  $\; b_n= \frac {\ln n }{1+\ln n}\;$ and $\;a_n= \frac 1{\set{1+\ln n}^2}.$ Further $V_k\in D_l(\Lambda) $ with the auxiliary function $\;\frac 1k v(x)=\frac 1k (1-x)^2\;$ and the norming constants   $\; b_n=\frac {\ln n+\ln \delta_1 }{k+\ln n+\ln \delta_1}\;$ and $\;a_n= \frac 1{\set{1+\ln n+\ln \delta_1}^2}.$

\item[(g)] Consider the Gamma distribution with DF satisfying $F^\prime(x)=\dfrac {x^\alpha e^{-x}}{\Gamma(\alpha+1)}, \quad x>0, \alpha>0.$ The $F\in D_l(\Lambda)$ with a constant auxiliary function $v(x)=1$  and norming constants $a_n=1, b_n= \ln n + \alpha \ln \ln n - \ln \Gamma(\alpha+1).$ Further $V_k\in D_l(\Lambda) $ with the auxiliary function $\;\frac 1k v(x)=\frac 1k\;$ and the norming constants $a_n=\dfrac 1k, b_n= \dfrac {\ln n+\ln \delta_1 }k + \alpha \ln( \ln n+\ln \delta_1) - \alpha\ln k- \ln \Gamma(\alpha+1).$\\ In particular when $\alpha=0$ we get exponential distribution with PDF $f(x)=e^{-x}, \quad x>0.$ The corresponding norming constants are $a_n=1,b_n=\ln n$ for $F$,and $a_n=\frac 1k,b_n=\frac {\ln n + \ln \delta_1} k$ for $V_k.$

\item[(h)]  For Log-Gamma distribution with df $F^\prime(x)=\frac {\alpha^\beta x^{-\alpha-1} (\ln x)^{\beta-1}}{\Gamma(\beta)},$ $\;F \in D_l(\Phi_\alpha)\;$  with norming constants $b_n=0$ and $\;a_n=((\Gamma(\beta))^{-1}(\ln n)^{\beta-1}n^{1/\alpha}.$ Further $\;V_k \in D_l(\Phi_{k\alpha})\;$  with norming constants $b_n=0$ and $b_n=0$ and $\;a_n=\set{\frac {(\ln n +\ln \delta_1)^{\beta-1} n^{\frac 1k}}{k^{\beta-1}\Gamma(\beta)(k!)^{\frac 1k}}}^{\frac 1\alpha}.$

\item[(i)]  For Cauchy distribution with df $F^\prime(x)=\frac 1{\pi(1+x^2)},$ $\;F \in D_l(\Phi_\alpha)\;$  with norming constants $b_n=0$ and $\;a_n=n/\pi.$ Further $\;V_k \in D_l(\Phi_{k\alpha})\;$  with norming constants $b_n=0$ and $\;a_n=\frac 1\pi \Set { n\delta_1}^{\frac 1k}.$

\item[(j)]   For Beta distribution with df $F(x)=\frac {x^{\alpha-1}(1-x)^{\beta-1}}{B(\alpha,\beta)},$ $\;F \in D_l(\Psi_\alpha)\;$  with norming constants $b_n=1$ and $\;a_n=\Set{n \frac {\Gamma(\alpha+\beta)}{\Gamma(\alpha)\Gamma(\beta+1)}}^{-1/\beta}.$ Further $\;V_k \in D_l(\Phi_{k\alpha})\;$  with norming constants $b_n=1$  and $\;\Set{ \Set{n\delta_1 }^{-\frac 1k} \frac {\Gamma(\alpha+\beta)}{\Gamma(\alpha)\Gamma(\beta+1)} }^{-\frac 1\beta}.$
\end{enumerate}

\section{Stochastic orderings}
In this section, we study stochastic orderings between the rvs having dfs $\;F, \, H_k\;$ and the limit dfs for the $\;k$-th random upper order statistic. We denote the rvs, by abuse of notation, by the dfs themselves. A rv $X$ is said to be stochastically smaller than $Y,$ denoted by $X \le_{\mbox{st}} Y$ if for all real $x,$ the df of $X,$ $P(X\ge x) \le P(Y\ge x),$ the df of $Y.$ 

 \begin{thm} The following are true. 
 \begin{enumerate}
 \item[(i)] $U_1 \le_{\mbox{st}} F.$
 \item[(ii)] $H_k \le_{\mbox{st}} H_{k+1},    F_k(x) \le_{\mbox{st}}  F_{k+1}(x),  R_k(x) \le_{\mbox{st}} R_{k+1}(x), T_k(x) \le_{\mbox{st}} T_{k+1}(x).$ 
 \item[(iii)] $ U_k(x) \le_{\mbox{st}} U_{k+1}(x)$ and $ U_k(x) \le_{\mbox{st}} F(x).$
 \item[(iv)] $R_k(x) \le_{\mbox{st}} H_k(x) \le_{\mbox{st}} F(x).$
 \item[(v)] $ F_k(x) \le_{\mbox{st}}  F(x)$ and $T_k(x) \le_{\mbox{st}} F^k(x).$ 
 \item[(vi)] $R_k(x) \le_{\mbox{st}} T_k(x) $
 \item[(vii)] $F_k(x) \le_{\mbox{st}} U_k(x) $
% \item[(viii)]  Conjecture: $F_k(x) \le_{\mbox{st}} R_k(x)$ fir all $x$ satisfying $\set{x: \ F(x)>2^{-(k-1)}}.$
\end{enumerate}\label{Thm_StOr}
\end{thm} 

\section{Some preliminary results and proofs of some main results}

The following lemma is useful for computations involving derivatives. 

\begin{lem} \label{lem:lr} Let $\;S,\;h\;$ be two real valued functions of a real variable and $\;S^{(r)}\;$ denote the $\;r$-th derivative of $\;S,\,r \geq 1,\;$ an integer. Whenever the derivatives exist, the following are true. 
\begin{enumerate} 
\item[(a)] If $\; S(x) =\dfrac {h(x)}{1-x}, x \ne 1,\;$ $\;S^{(r)}(x)=r! \sum_{i=0}^r \dfrac {h^{(i)}(x)}{i!(1-x)^{r-i+1}}.\;$  
\item[(b)] If $\;S(x)=h(x)e^x,\;$ $\;S^{(r)}(x)=e^x \sum\limits_{i=0}^r \binom ri h^{(i)}(x).$
\item[(c)] If $\;S(x)=\dfrac {x^m}{(1-x)^n},$ $x\ne 1,\;$ for some integers $\;m \ge n \ge 1,\;$ $\;S^{(r)}(x)=r!\sum_{i=0}^r \binom m{r-i} \binom {n+i-1}i \dfrac {x^{m-r+i}}{(1-x)^{n+i}}.$ \end{enumerate}
\end{lem}

\begin{proof} 
(a) By differentiating $\;(1-x)S(x)=h(x)\;$ repeatedly $\;r\;$ times with respect to $\;x,\;$ we get $\;(1-x)S^{(r)}(x)- r S^{(r-1)}(x)=h^{(r)}(x) \Rightarrow S^{(r)}(x) = \frac {h^{(r)}(x)}{1-x}+\frac{r}{1-x}S^{(r-1)}(x).$ Hence 
\begin{eqnarray*}
S^{(r)}(x)&=&\frac {h^{(r)}(x)}{1-x}+\dfrac{r}{1-x}\left(\dfrac{h^{(r-1)}(x)}{1-x}+\dfrac{r-1}{1-x}S^{(r-2)}(x)\right),
\\&=&\dfrac {h^{(r)}(x)}{1-x}+\dfrac{rh^{(r-1)}(x)}{(1-x)^2}+\dfrac{r(r-1)}{(1-x)^2}\left(\frac{h^{(r-2)}(x)}{1-x}+\dfrac{r-2}{1-x}S^{(r-3)}(x)\right),
\\&=&\dfrac {r!h^{(r)}(x)}{r!(1-x)}+\dfrac{r!h^{(r-1)}(x)}{(r-1)!(1-x)^2}+\dfrac{r!h^{(r-2)}(x)}{(r-2)!(1-x)^3}+\ldots + \dfrac{r!h^{(1)}(x)}{1!(1-x)^r}+\dfrac{r!h(x)}{0!(1-x)^{r+1}}, \\ & & \text{proving (a). }
\end{eqnarray*}

(b) The proof is by induction on $\;r.\;$ For $\;r=1,$ observe that $\; S^{(1)}(x)=e^x(h(x)+ h^{(1)}(x))=e^x \sum_{i=0}^1 \binom 1i h^{(i)}(x), \;\; \mbox{where} \; h^{(0)}(x) = h(x),$  which is true.
Assuming that the result is true for $\;r=s,\;$ that is, $\;S^{(s)}(x)=e^x \sum\limits_{i=0}^s \binom si h^{(i)}(x),\;$ and differentiating again with respect to $\;x,\;$ we have 
\begin{eqnarray*}
S^{(s+1)}(x)&=&e^x \sum_{i=0}^s \binom si h^{(i)}(x)+e^x \sum_{i=0}^s \binom si h^{(i+1)}(x),
\\&=&e^x \sum_{i=0}^s \binom {s+1}i \left(\dfrac {s+1-i}{s+1}\right) h^{(i)}(x)+e^x \sum_{i=0}^s \binom {s+1}{i+1}\left( \dfrac {i+1}{s+1}\right)h^{(i+1)}(x),
\\&=&e^x \sum_{j=0}^{s+1} \binom {s+1}j \left(\dfrac {s+1-j}{s+1}\right) h^{(j)}(x)+e^x \sum_{j=0}^{s+1} \binom {s+1}j\left( \dfrac j{s+1}\right) h^{(j)}(x),
\\&=&e^x \sum_{j=0}^{s+1} \binom {s+1}j  h^{(j)}(x), \;\;\text{which completes the proof of (b).}
\end{eqnarray*}

(c) The proof is by induction on $\;r.\;$ For $\;r=1,\;$ observe that  $\; S^{(1)}(x)=\dfrac{mx^{m-1}}{(1-x)^n}+\dfrac {nx^m}{(1-x)^{n+1}} = \sum_{i=0}^1 \binom m{1-i} \binom {n+i-1}i \dfrac {x^{m-1+i}}{(1-x)^{n+i}},$ which is true.
Assuming that the result is true for $\;r=s,\;$ that is, $\;S^{(s)}(x)=s!\sum\limits_{i=0}^s \binom m{s-i} \binom {n+i-1}i \dfrac {x^{m-s+i}}{(1-x)^{n+i}},\;$ and differentiating again with respect to $\;x,\;$ we have 
\begin{eqnarray*}
&& S^{(s+1)}(x)= s!\sum_{i=0}^s \binom m{s-i} \binom {n+i-1}i \left( \dfrac {(m-s+i) x^{m-s+i-1}}{(1-x)^{n+i}}+\dfrac {(n+i) x^{m-s+i}}{(1-x)^{n+i+1}}\right),
\\&=& s!\sum_{i=0}^s \binom m{s-i+1} \binom {n+i-1}i \dfrac {(s+1-i) x^{m-s-1+i}}{(1-x)^{n+i}}+ s!\sum_{i=0}^s \binom m{s-i} \binom {n+i}{i+1}  \dfrac {(i+1) x^{m-s+i}}{(1-x)^{n+i+1}},
\\&=& s!\sum_{j=0}^{s+1} \binom m{s+1-j} \binom {n-1+j}j \dfrac {(s+1-j) x^{m-s-1+j}}{(1-x)^{n+j}}+ s!\sum_{j=0}^{s+1} \binom m{s+1-j} \binom {n-1+j}j \dfrac {j x^{m-s-1+j}}{(1-x)^{n+j}},
\\&=& (s+1)! \sum_{j=0}^{s+1} \binom m{s+1-j} \binom {n-1+j}j \dfrac {j x^{m-s-1+j}}{(1-x)^{n+j}}, \;\;\text{proving (c).}
\end{eqnarray*} 
\end{proof} 

\begin{rem} We have the following particular cases. 
\begin{enumerate}
\item[(i)] In (a) of Lemma \ref{lem:lr}, if $\;h(x)=x^m\;$ for some fixed positive integer $\;m\ge 1,\;$ we have $\; h^{(l)}(x)=\frac{m!}{(m-l)!} x^{m-l}\;$ and $\; S^{(r)}(x)=r! \sum_{l=0}^r \frac{m! x^{m-l}}{l!(m-l)!(1-x)^{r+1-l}}=r! \sum_{l=0}^r \binom mr\frac {x^{m-l}}{(1-x)^{r+1-l}}.$
\item[(ii)] In (b) of Lemma \ref{lem:lr}, if $\;h(x)=x^m\;$ then $\;S^{(r)}(x)=e^x \sum\limits_{l=0}^r \binom rl\frac {m!}{(m-l)!} x^{m-l}.$
\end{enumerate}
\end{rem}

\begin{proof}[\textbf {Proof of Theorem \ref{lem:1}:}] For real $\;x < y,$  $F(x)\le F(y)\;$ implies that  $\;\set{1-F(x)}^k \ge \set{1-F(y)}^k,\;$ so that $\;H_k(x) \le H_k(y).\;$ Hence $\;H_k\;$ is non-decreasing.
Since $\;F\;$ is right-continuous,  $\;H_k\;$ is right-continuous. Also, $\; \lim\limits_{x\rightarrow-\infty} H_k(x)=0,\;$ and $\;\lim\limits_{x\rightarrow \infty} H_k(x)=1. \;$ 
Further, $\;H_k\;$ is also absolutely continuous with pdf $\; h_k(x)=H_k^\prime(x)=k\set{1-F(x)}^{k-1} f(x), x \in \Real.$
Observe that $\;H_k(r(F))=1-(1-F(r(F)))^k=1\;$ so that $\;r(H_k)\le r(F).$ If possible let $\;r(H_k)<r(F).\;$ Then $1=H_k(r(H_k))=1-(1-F(r(H_k)))^k<1\;$ a contradiction so that $\;r(H_k)= r(F).$ Now we prove (a), (b) and (c). 
\begin{enumerate}
\item[(a)] If $\;F\in D_l(\Phi_\alpha),\;$ then by (a) of Theorem $\ref{T-l-max},$ $\; 1-F\in RV_{-\alpha} \Rightarrow \lim_{t \rightarrow \infty}\dfrac {1-F(tx)}{1-F(t)} = x^{-\alpha}.\;$ Further, $\;\lim_{t \rightarrow \infty}\dfrac {1-H_k(tx)}{1-H_k(t)}=$ $\lim_{t \rightarrow \infty}\left(\dfrac {1-F(tx)}{1-F(t)}\right)^k$ $ = x^{-k\alpha}. \;$ Hence $\;1-H_k\in RV_{-k\alpha}\;$ which implies that $\;H_k\;$ belongs to $\;D_l(\Phi_{k\alpha}).\;$ Observing that $\;1-H_k=(1-F)^k,\;$ 
$\; \left(\dfrac 1{1-H_k}\right)^- (n)= \left(\dfrac 1{(1-F)^k}\right)^-(n)=\left(\dfrac 1{1-F}\right)^- \left(n^{\frac 1k}\right),\;$ we obtain the norming constants, proving (a). 
\item[(b)] If $\;F\in D_l(\Psi_\alpha),\;$ then $\;\displaystyle 1-F\left(r(F)-\frac 1x\right) \in RV_{-\alpha} \Rightarrow \lim_{t \rightarrow \infty}\dfrac{1-F\left(r(F)-\dfrac 1{tx}\right)}{1-F\left(r(F)-\dfrac 1t\right)} = x^{-\alpha}.\;$  Further, $\;r(H_k)=r(F).\;$ It follows that  
$\;\lim_{t \rightarrow \infty}\frac {1-H_k\left(r(H_k)-\frac 1{tx}\right)}{1-H_k\left(r(H_k)-\frac 1t\right)}$ \\ $=\lim_{t \rightarrow \infty} \left(\frac {1-F\left(r(F)-\frac 1{tx}\right)}{1-F\left(r(F)-\frac 1t\right)}\right)^k= x^{-k\alpha}.$
The rest of the proof is on lines similar to the proof in (a) and is omitted. 
\item[(c)] Since $\;F\in D_l(\Lambda),\;$ there exist a constant $\;c>0,\;$ continuous, positive function $\;v,\;$ with $\;\lim\limits_{t\uparrow r(F) } v^\prime(t)=0\;$ such that for all $\;t \in (z, r(F)),\;$ with $\;z < r(F),\;$  $\;F\;$ is a von-Mises function which has a representation $\displaystyle 1-F(x)=c\exp \left(-\int_z^x \frac 1{v(t)} dt\right),\;$ and $\;F\;$  satisfies the condition $\;\displaystyle \lim_{x\rightarrow \infty}\dfrac{\left(1-F(x)\right)F^{\prime\prime}(x)}{\left(F^\prime (x)\right)^2}=-1.\;$ We now show that $\;H_k\;$ is also a von-Mises function and use (e) of Theorem $\ref{T-l-max}$ to conclude that $\;H_k \in D_l(\Lambda).\;$ \\
Since $H_k^\prime(x)=k\left(1-F(x)\right)^{k-1}F^\prime(x),$ and $H_k^{\prime\prime}(x)=k\left(1-F(x)\right)^{k-1}F^{\prime\prime}(x)-k(k-1)\left(1-F(x)\right)^{k-2} \left(F^\prime(x)\right)^2,$ it follows that  \[\lim_{x\rightarrow \infty}\dfrac{(1-H_k(x))H_k^{\prime\prime}(x)}{(H_k^\prime (x))^2}=\lim_{x\rightarrow \infty} \frac 1k \left(\frac{(1-F(x))F^{\prime\prime}(x)}{(F^\prime (x))^2}-k+1\right) = -1,\] and hence from (\ref{von-Mises}), $\;H_k\;$ is a von-Mises function and belongs to $\;D_l(\Lambda)\;$ with norming constants $\;a_n=\dfrac 1k v(b_n)\;$ and $\;b_n=\left(\dfrac 1{1-H_k}\right)^-(n),\;$ proving (c).  
\end{enumerate}
\end{proof}

\begin{proof}[\textbf{Proof of Theorem \ref{Thm-Fk}:}]
\begin{enumerate}
\item[(a)] By Theorem \ref{thm:r-max-df}, $\;F_k\;$ is a df with pdf $\;f_k(x)=$ \\ $\frac {f(x)}{(k-1)!} (-\ln F(x))^{k-1}.\;$ To show that $\;r(F) = r(F_k),\;$ we consider two cases. First, if $\;r(F)=\infty,\;$ then $\;\lim_{x\rightarrow r(F)} F_k(x)=\lim_{x\rightarrow r(F)} F(x)\sum_{r=0}^{k-1} \frac {(-\ln F(x))^r}{r!}=1 \Rightarrow r(F_k) \le r(F).\;$ If possible, let $\;r(F_k)<\infty.\;$ Then $\;1 =F_k(r(F_k))$ \\  $=F(r(F_k))\sum_{i=0}^{k-1} \frac {(-\ln F(r(F_k)))^i}{i!}$ $ <F(r(F_k))\sum_{i=0}^\infty \frac {(-\ln F(r(F_k)))^i}{i!}$ \\ $=F(r(F_k))\exp \set{-\ln F(r(F_k))}=1,\;$ a contradiction proving that $\; r(F_k)=\infty.\;$ If $\;r(F)<\infty,\;$ arguing as above, if possible, let $\;r(F_k)<r(F).\;$ Since $\;F(r(F_k))=1,\;$ we must have $\;r(F_k)=r(F)-\epsilon,\;$ for some $\;\epsilon >0.\;$ Repeating steps as earlier we get a contradiction proving that $\;\epsilon \;$ must be equal to 0 and hence $\;r(F_k) = r(F).\;$ 

Defining $\; h(x)=\dfrac {-\ln F(x) }{1-F(x)},\;$ we have $\; \lim_{x\rightarrow r(F)} h(x)=1\;$ by L'Hospital's rule, and $\; \lim_{x \rightarrow r(F)}\frac {1-F_k(x)}{(1-F(x))^k}$ $=\lim_{x \rightarrow r(F)}\frac {-f_k(x)}{-k(1-F(x))^{k-1}f(x)}$ $
= \lim_{x \rightarrow r(F)}\frac {-\frac {f(x)}{(k-1)!} (-\ln F(x))^{k-1}}{-k(1-F(x))^{k-1}f(x)}	=$ $ \lim_{x \rightarrow r(F)} \frac {h^{k-1}(x)}{k!}=$ $\frac 1{k!},\;$ showing that $\;F_k\;$ is tail equivalent to $\;H_k.$
\item[(b)] If $\;F\in D_l(\Phi_\alpha),\;$ then by (a) of Theorem $\ref{T-l-max},$ $\; 1-F\in RV_{-\alpha} \Rightarrow \lim_{t \rightarrow \infty}\dfrac {1-F(tx)}{1-F(t)} = x^{-\alpha}.\;$ Further, $\;r(F)=r(F_k)=\infty,\;$ and using (a), and multiplying and dividing by $\;\left(\dfrac {1-F(tx)}{1-F(t)}\right)^k,\;$ we get $\;\lim_{t \rightarrow \infty}\dfrac {1-F_k(tx)}{1-F_k(t)}=$ $\lim_{t \rightarrow \infty}\left(\dfrac {1-F(tx)}{1-F(t)}\right)^k =$ $ x^{-k\alpha}.\;$ 
Hence $\;1-F_k \in RV_{-k\alpha},\;$ which implies that $\;F_k\;$ belongs to $\;D_l(\Phi_{k\alpha}).\;$ Observing that $\;1-H_k=(1-F)^k\;$ and $\;1-F_k \sim \dfrac 1{k!} (1-H_k)=\dfrac 1{k!}  (1-F)^k,\;$ and 
$\; \left(\dfrac 1{1-H_k}\right)^- (n)= \left(\dfrac 1{(1-F)^k}\right)^-(n)=\left(\dfrac 1{1-F}\right)^- \left(n^{\frac 1k}\right),\;$ and $\;\left(\dfrac 1{1-F_k}\right)^- (n)=\left(\dfrac {k!}{(1-F)^k}\right)^- (n)=\left(\dfrac 1{1-F}\right)^- \left(\dfrac n{k!}\right)^{\frac 1k}, $ we obtain the norming constants, proving (b). 
\item[(c)] If $\;F\in D_l(\Psi_\alpha),\;$ then by (b) of Theorem $\ref{T-l-max},$  $\;\displaystyle 1-F\left(r(F)-\frac 1x\right) \in RV_{-\alpha} \Rightarrow \lim_{t \rightarrow \infty}\dfrac{1-F\left(r(F)-\dfrac 1{tx}\right)}{1-F\left(r(F)-\dfrac 1t\right)} = x^{-\alpha}.\;$  Further, $\;r(F_k)=r(F)\;$ and the rest of the proof is on lines similar to the proof in (b) and is omitted. 
\item[(d)] We have $\;\lim_{x\rightarrow \infty} \dfrac{1-F_k(x)}{1-H_k(x)} = \lim_{x\rightarrow \infty} \dfrac {1-F_k(x)}{\set{1-F(x)}^k} = \frac 1{k!}.$
Thus $\;F_k\;$ and $\;H_k\;$ are tail equivalent and hence by proof of (c) of Theorem \ref{lem:1}, $\;F_k \in D_l(\Lambda)\;$ and the function $\;\dfrac 1k v(t)\;$ satisfies all the conditions for it to be an auxiliary function, proving (d).  
\end{enumerate}
\end{proof}

\begin{proof}[\textbf{Proof of Theorem \ref{Thm_DU1}:}]  
Setting the operator $D^r=\dfrac{d^r}{dF^r(x)}, r=1,2,3,\ldots,\;$ and considering the term corresponding to $\;i=0\;$ in (\ref{eqn:gr-max1}), we have $\; \frac 1n \sum_{r=m+1}^{m+n} F^r(x) = \frac {F^{m+1}(x)-F^{m+n+1}(x)}{n(1-F(x))}.\;$ Applying Lemma \ref{lem:lr}, we get 
$\;D^i \Set{\frac {F^{m+1}(x)-F^{m+n+1}(x)}{1-F(x)}} =$ \\ $i! \sum_{l=0}^i\set{\binom {m+1}l \frac {F^{m+1-l}(x)}{(1-F(x))^{i+1-l}}- \binom {m+n+1}l \frac {F^{m+n+1-l}(x)}{(1-F(x))^{i+1-l}}}.\;$ For the general $\;i$th term, for $\;i\ge 1,\;$ we have $\;\frac 1n \sum_{r=m+1}^{m+n} \binom ri  F^{r-i}(x) \set{1-F(x)}^i = \frac {\set{1-F(x)}^i}{i!n} \sum_{r=m+1}^{m+n} \frac {r!}{(r-i)!}  F^{r-i}(x) $
\begin{eqnarray*}
&=& \frac {\set{1-F(x)}^i}{i!n} \sum_{r=m+1}^{m+n} D^i F^r(x) =\frac {\set{1-F(x)}^i}{i!n} D^i \set{\sum_{r=m+1}^{m+n} F^r(x)}
\\&=& \frac {\set{1-F(x)}^i}{i!n} D^i \Set{\frac {F^{m+1}(x)-F^{m+n+1}(x)}{1-F(x)}}
\\&=&\frac {\set{1-F(x)}^i}{i!n} i! \sum_{l=0}^i\set{\binom {m+1}l \frac {F^{m+1-l}(x)}{(1-F(x))^{i+1-l}}- \binom {m+n+1}l \frac {F^{m+n+1-l}(x)}{(1-F(x))^{i+1-l}}}
\end{eqnarray*}
\begin{eqnarray*}
&=&\frac 1n \sum_{l=0}^i \set{ \binom {m+1}l F^{m+1-l}(x)- \binom {m+n+1}l F^{m+n+1-l}(x)}\set{1-F(x)}^{l-1} 
\\&=& \frac {F^{m+1}(x)-F^{m+n+1}(x)}{n(1-F(x))}+ \sum_{l=1}^i \binom {m+1}l F^{m+1-l}(x) \frac {\set{n(1-F(x))}^{l-1}}{n^l} -\\ && \sum_{l=1}^i \frac {(m+n+1)(m+n)\cdots(m+n+2-l)}{n^l l!} F^{m+n+1-l}(x) \set{n(1-F(x))}^{l-1}.
\end{eqnarray*}
Substituting these in (\ref{eqn:gr-max1}), we get 
\begin{eqnarray*}
F_{k:N_n}&=& \frac {F^{m+1}(x)-F^{m+n+1}(x)}{n(1-F(x))}+\sum_{i=1}^{k-1} \frac {F^{m+1}(x)-F^{m+n+1}(x)}{n(1-F(x))}+\\&& \sum_{i=1}^{k-1}  \sum_{l=1}^i \binom {m+1}l F^{m+1-l}(x) \frac {\set{n(1-F(x))}^{l-1}}{n^l} -\\&&\sum_{i=1}^{k-1} \sum_{l=1}^i \frac 1{l!} \prod_{j=-1}^{l-2} \frac {m+n-j}n  F^{m+n+1-l}(x) \set{n(1-F(x))}^{l-1}.
\end{eqnarray*}
Replacing $\;x\;$ by $\;a_n x+b_n\;$ and taking limit as $\;n \rightarrow \infty\;$  we have 
\begin{eqnarray*}
J_k(x)&=&\lim_{n\to  \infty} F_{k:N_n}(a_nx+b_n)
\\&=&\frac {1-G(x)}{-\ln G(x)}+\sum_{i=1}^{k-1} \frac {1-G(x)}{-\ln G(x)}+\sum_{i=1}^{k-1}\sum_{l=1}^i 0+\sum_{i=1}^{k-1} \sum_{l=1}^i \frac 1{l!} G(x) 
(-\ln G(x))^{l-1}
\\&=& k\Set{\frac {1-G(x)}{-\ln G(x)}}-(k-1)G(x)-G(x) \sum_{l=2}^{k-1} (k-l) \frac {(-\ln G(x))^{l-1}}{l!},
\end{eqnarray*}
completing the proof.
\end{proof}

\begin{proof}[\textbf{Proof of Theorem \ref{Thm-FkW}:}] 
Clearly $\lim_{x\to  +\infty} U_1(x)=$ $\lim_{x\to  +\infty} \frac {1-F(x)}{-\ln F(x)}=
$ $\frac {-f(x)}{-\frac {f(x)}{F(x)}} =1$ and \\
$ \lim_{x\to  -\infty} U_1(x)=\lim_{x\to  -\infty} \frac {1-F(x)}{-\ln F(x)}=0.$ Further $U_1$ is right continuous since $F$ is. Differentiating with respect to $x$ we get 
$u_1(x)=U_1^\prime(x)=\frac {-\ln F(x) (-f(x))-(1-F(x))\frac {-f(x)}{F(x)}}{(-\ln F(x))^2}=$ $\frac {-f(x)F(x)+f(x)U_1(x)}{-F(x) \ln F(x)}$ $=\frac {f(x)}{F(x)}\frac {U_1(x)-F(x)}{-\ln F(x)}.$ 
Now we claim that $U_1(x) \ge F(x).$ If the claim is true, then  $\;\frac {1-F(x)}{F(x)} \ge -\ln F(x) \Rightarrow \frac 1{F(x)} -1 \ge \ln \frac 1{F(x)}.$  Setting $\;y=\frac 1{F(x)}-1,\;$ since $F(x)\le 1$ it follows that $y\ge 0.$ Then from the preceding expression it follows that $\; y\ge \ln (1+y) \Rightarrow e^y \ge 1+y \;$ which is true and hence the claim is true. Thus $U_1$ is non-decreasing and so is a df. 

Observe that $U_1(r(F))=\lim_{x\to  r(F)} \frac {1-F(x)}{-\ln F(x)}=1$ so that $r(U_1)\le r(F).$ If possible, let $r(U_1)< r(F).$ Then $0<F(r(U_1))<1$ and then using Taylor's expansion we have $\;-\ln (F(r(U_1)))=$ $ -\ln (1-(1-F(r(U_1))))=$ $(1-F(r(U_1)))+\frac {(1-F(r(U_1)))^2}{2}+\ldots > (1-F(r(U_1))).\;$
Hence $U_1(r(U_1))<1$  which is a contradiction, proving $\;r(U_1)=r(F).$ We then have $\; \lim_{x\to  \infty} \frac {1-U_1(x)}{1- F(x)}=$ $ \lim_{x\to  \infty} \frac {U_1^\prime(x)}{f(x)}=$ $\lim_{x\to  \infty} \frac 1{F(x)} \Set{\frac  {1-F(x)+F(x) \ln F(x)}{(1-F(x))^2}}$ $=\frac 12,\;$ proving the theorem. 
\end{proof}

\begin{proof}[\textbf{Proof of Theorem \ref{Thm-FkDU}:}] 
\begin{enumerate}
\item[(a)]  Clearly $\lim_{x\to  +\infty}U_k(x)=1$ and $\lim_{x\to  -\infty}U_k(x)=0.$ Further $U_k$ is right continuous since $F$ is. Differentiating, we get 
\begin{eqnarray*}
u_k(x) &=& ku_1 (x)-f(x) \sum_{l=1}^{k-1} (k-l) \frac {(-\ln F(x))^{l-1}}{l!}+f(x)\sum_{l=1}^{k-1} (k-l) \frac {(l-1)(-\ln F(x))^{l-2}}{l!}
\\&=&  ku_1 (x)-f(x) \sum_{l=2}^k (k-l+1) \frac {(-\ln F(x))^{l-2}}{(l-1)!}+f(x)\sum_{l=1}^{k-1} (k-l) \frac {(l-1)(-\ln F(x))^{l-2}}{l!}
\\&=&  ku_1 (x)-f(x) \sum_{l=2}^k l(k-l+1) \frac {(-\ln F(x))^{l-2}}{l!}+f(x)\sum_{l=2}^k (k-l)(l-1) \frac {(-\ln F(x))^{l-2}}{l!}
\\&=&  \frac {kf(x)}{F(x)}\Set{\frac {1-F(x)+F(x) \ln F(x)}{(-\ln F(x))^2}} - k f(x) \sum_{l=2}^k  \frac {(-\ln F(x))^{l-2}}{l!}
\\&=& \frac {kf(x)}{(-\ln F(x))^2}\Set{\frac 1{F(x)} - \sum_{l=0}^k  \frac {(-\ln F(x))^l}{l!}} >0, \mbox{since} \; \sum_{l=0}^k  \frac {(-\ln F(x))^l}{l!}<e^{-\ln F(x)}=\frac 1{F(x)}
\end{eqnarray*}
Hence $U_k$ is non-decreasing and hence is a df. Observe that  $\; \lim_{x\to  r(F)} U_k(x)=\lim_{x\to  r(F)} \set{ kU_1(x)- (k-1)F(x) -F(x) \sum_{l=2}^{k-1} (k-l) \frac {(-\ln F(x))^{l-1}}{l!}}=1,\;$ so that $r(U_k)\le r(F).$ If possible, let $\;r(U_k)<r(F).\;$  For convenience set $\xi=-\ln F(r(U_k)).$ Then it follows that $e^{-\xi}= F(r(U_k))$ and for $F(r(U_k))<1$ we have $\xi>0,$ and  
\begin{eqnarray*}
1 = U_k(r(U_k)) &=& k\Set{\frac {1-F((r(U_k)))}{-\ln F((r(U_k)))}} -F(r(U_k)) \sum_{l=1}^{k-1} (k-l) \frac {(-\ln F(r(U_k)))^{l-1}}{l!} 
\\&=& k\Set{\frac {1-e^{-\xi}}{\xi}}-\sum_{l=1}^{k-1} (k-l) \frac {e^{-\xi} \xi^{l-1}}{l!} 
= k\sum_{j=1}^\infty \frac{e^{-\xi} \xi^{j-1}}{j!}-\sum_{l=1}^{k-1} (k-l) \frac {e^{-\xi} \xi^{l-1}}{l!} 
\\&=& k\sum_{j=k}^\infty \frac{e^{-\xi} \xi^{j-1}}{j!}+\sum_{l=1}^{k-1} \frac {e^{-\xi} \xi^{l-1}}{(l-1)!} 
=k\sum_{j=k}^\infty \frac{e^{-\xi} \xi^{j-1}}{j!}+1-\sum_{l=k}^\infty \frac {e^{-\xi} \xi^{l-1}}{(l-1)!} 
\\&=& k\sum_{j=k+1}^\infty \frac{e^{-\xi} \xi^{j-1}}{j!}+\frac{e^{-\xi} \xi^{k-1}}{(k-1)!} +1-\sum_{l=k+1}^\infty \frac {e^{-\xi} \xi^{l-1}}{(l-1)!}- \frac {e^{-\xi} \xi^{k-1}}{(k-1)!} \\
&=& 1+ k\sum_{j=k+1}^\infty \frac{\xi^j}{j!} - \sum_{j=k+1}^\infty \frac {\xi^j}{(j-1)!} =1+ \sum_{j=k+1}^\infty \frac{\xi^j}{(j-1)!}\Set{\frac kj-1} <1, 
\end{eqnarray*} a contradiction as $j>k,$ proving that $r(U_k) = r(F). $
\item[(b)] Observe that $\lim_{x\to  \infty} \frac {1-U_k(x)}{1-H_k(x)} $
\begin{eqnarray*}
&=&\lim_{x\to  \infty} \frac {1-kU_1(x)+F(x) \sum\limits_{l=1}^{k-1} (k-l) \dfrac {(-\ln F(x))^{l-1}}{l!}}{(1-F(x))^k} 
\\&=& \lim_{x\to  \infty} \frac {1-k\Set{\frac {1-F(x)}{-\ln F(x)}}+F(x) \sum\limits_{l=1}^{k-1} (k-l) \dfrac {(-\ln F(x))^{l-1}}{l!}}{(1-F(x))^k} 
\\&=& \lim_{x\to  \infty}\Set{\frac {1-F(x)}{-\ln F(x)}}  \frac {-\ln F(x) -k(1-F(x)) +F(x) \sum\limits_{l=1}^{k-1} (k-l) \dfrac {(-\ln F(x))^l}{l!}}{(1-F(x))^{k+1}} 
\\&=& \lim_{x\to  \infty} \frac {-\frac {f(x)}{ F(x)} +k f(x) + f(x) \sum\limits_{l=1}^{k-1} (k-l) \dfrac {(-\ln F(x))^l}{l!}- f(x) \sum\limits_{l=1}^{k-1} (k-l) \dfrac {(-\ln F(x))^{l-1}}{(l-1)!}}{-(k+1)f(x) (1-F(x))^k} 
\\&=& \lim_{x\to  \infty} \frac {-\frac 1{ F(x)}+\sum\limits_{l=0}^{k-1} (k-l) \dfrac {(-\ln F(x))^l}{l!}- \sum\limits_{l=0}^{k-1} (k-l-1) \dfrac {(-\ln F(x))^l}{l!}}{-(k+1)  (1-F(x))^k}
\\&=& \lim_{x\to  \infty} \frac 1{F(x)} \frac {1-F(x)\sum\limits_{l=0}^{k-1} \dfrac {(-\ln F(x))^l}{l!}}{(k+1)  (1-F(x))^k}
= \lim_{x\to  \infty} \frac 1{(k+1) F(x)} \lim_{x\to  \infty} \frac {1-F_k(x)}{(1-F(x))^k}
=\frac 1{(k+1)!}
\end{eqnarray*}
\end{enumerate}
The rest of proof is same as in Theorem \ref{Thm-Fk} except that $\;k!\;$ is replaced by  $\;(k+1)!\;$ 
\end{proof}

\begin{proof}[\textbf{Proof of Theorem \ref{Thm_FkB}:}] \begin{itemize}
\item[(a)]
 Let $\;D^r=\dfrac{d^r}{d(p_n F(x))^r},\, r=1,2,3,\ldots.\;$ Considering the term corresponding to $\;i=0\;$ in (\ref{eqn:gr-max1}), we have 
\begin{eqnarray*} \sum_{r=m}^{m+n}  F^r(x)  \binom n{r-m}p_n^{r-m}q_n^{m+n-r} & = & F^m(x) \sum_{r=m}^{m+n} \binom n{r-m}(p_nF(x))^{r-m}q_n^{m+n-r}, \\ & = & F^m(x) (p_nF(x)+q_n)^n. \end{eqnarray*}
By Leibnitz rule for derivative of the product, we have 
\[ D^i \left((p_n F (x))^m (p_nF(x)+q_n)^n\right) = i! \sum_{l=0}^i \binom ml \binom n{i-l}(p_n F (x))^{m-l} (p_nF(x)+q_n)^{n-i+l}.  \]
For $\;1 \leq i \leq k-1, \;$ in (\ref{eqn:gr-max1}), we now have 
\begin{eqnarray*}
& &\sum_{r=m}^{m+n}\binom ri  F^{r-i}(x) (1-F(x))^i \binom n{r-m}p_n^{r-m}q_n^{m+n-r}
\\&=& \dfrac{F^{m-i}(x)(1-F(x))^i}{i!} \sum_{r=m}^{m+n} \binom n{r-m} \dfrac {r!}{(r-i)!} (p_n F (x))^{r-m}q_n^{m+n-r},
\\&=& \dfrac{F^{m-i}(x)(1-F(x))^i}{i!} (p_n F (x))^{i-m}\sum_{r=m}^{m+n} \binom n{r-m} \dfrac {r!}{(r-i)!} (p_n F (x))^{r-i}q_n^{m+n-r},
\\&=& \dfrac{F^{m-i}(x)(1-F(x))^i}{i!} (p_n F (x))^{i-m}\sum_{r=m}^{m+n} \binom n{r-m} D^i (p_n F (x))^r q_n^{m+n-r},
\\&=& \dfrac{F^{m-i}(x)(1-F(x))^i}{i!} (p_n F (x))^{i-m} D^i \sum_{r=m}^{m+n} \binom n{r-m} (p_n F (x))^r q_n^{m+n-r},
\end{eqnarray*}
\begin{eqnarray*}
&=& \dfrac{F^{m-i}(x)(1-F(x))^i}{i!} (p_n F (x))^{i-m} D^i \left( (p_n F (x))^m (p_nF(x)+q_n)^n\right).
\\&& F^{m-i}(x)(1-F(x))^i \sum_{l=0}^i \binom ml \binom n{i-l}( p_n F (x))^{i-l} (p_nF(x)+q_n)^{n-i+l}
\end{eqnarray*}
Substituting all these expressions, we get 
\begin{eqnarray*}
F_{k:N_n}(x) &=& F^m(x) (p_nF(x)+q_n)^n+\sum_{i=1}^{k-1} (1-F(x))^i \binom ni p_n^i F^m(x) (p_nF(x)+q_n)^{n-i}\\&&+\sum_{i=1}^{k-1} (1-F(x))^i \sum_{l=1}^i \binom ml \binom n{i-l}p_n^{i-l}F^{m-l}(x)(p_nF(x)+q_n)^{n-i+l},  
\end{eqnarray*}
\begin{eqnarray*}
\\&=&  F^m(x) (p_nF(x)+q_n)^n+\sum_{i=1} ^{k-1} (1-F(x))^i \dfrac {n!}{(n-i)!i!} p_n^i F^m(x) (p_nF(x)+q_n)^{n-i} 
\\&& +\sum_{i=1}^{k-1} (1-F(x))^i \sum_{l=1}^i \binom ml \dfrac {n!}{(i-l)!(n-i+l)!} p_n^{i-l}F^{m-l}(x)(p_nF(x)+q_n)^{n-i+l},  
\\&=&  F^m(x) (p_nF(x)+q_n)^n+\sum_{i=1}^{k-1} \dfrac {(n(1-F(x)))^i}{i!} \left(\prod_{j=0}^{l-1} \dfrac {n-j}n\right) p_n^i F^m(x) (p_nF(x)+q_n)^{n-i} 
\\&& +\sum_{i=1}^{k-1} \sum_{l=1}^i \dfrac {(n(1-F(x)))^i}{n^l(i-l)!} \binom ml \left(\prod_{j=0}^{i-l-1}  \dfrac {n-j}{n}\right) p_n^{i-l}F^{m-l}(x)(p_nF(x)+q_n)^{n-i+l}.  
\end{eqnarray*}
Replacing $\;x\;$ by $\;a_nx + b_n\;$ above and using the facts that $\;  \lim_{n \rightarrow \infty}p_n=1$  and 
 \[ \lim_{n \rightarrow \infty} (q_n+p_nF(a_n x + b_n))^n=\lim_{n \rightarrow \infty} \left(1-\dfrac {np_n (1-F(a_n x + b_n))}{n}\right)^n=e^{-(-\ln G(x))}=G(x), \] we get 
$\; F_k(x) = G(x) +G(x) \sum_{i=1}^{k-1} \frac {(-\ln G(x))^i}{i!} +\sum_{i=1}^{k-1} \sum_{l=1}^i 0 = G(x) \sum_{i=0}^{k-1} \dfrac {(-\ln G(x))^i}{i!}.$
This completes the proof of (a). 
\item[(b)] Let $\;D^r=\dfrac{d^r}{d(\lambda_n F(x))^r},\, r=1,2,3,\ldots.\;$ Considering the term corresponding to $\;i=0\;$ in (\ref{eqn:gr-max1}), we have 
\[ \sum_{r=m}^\infty F^r(x)\frac {e^{-\lambda_n} \lambda_n^{r-m}}{(r-m)!} = e^{-\lambda_n} F^m(x) \sum_{r=m}^\infty \frac{(\lambda_n F (x))^{r-m}}{(r-m)!} =e^{-\lambda_n(1-F(x)} F^m(x). \]
For $\;1 \leq i \leq k-1, \;$ in (\ref{eqn:gr-max1}), we now have 
\begin{eqnarray*}
& &\sum_{r=m}^\infty \binom ri  F^{r-i}(x)(1-F(x))^i \dfrac {e^{-\lambda_n} \lambda_n^{r-m}}{(r-m)!}, 
\\&=& \dfrac{e^{-\lambda_n}F^{m-i}(x)(1-F(x))^i}{i!} \sum_{r=m}^\infty \dfrac {r!}{(r-i)!} \dfrac{(\lambda_n F (x))^{r-m}}{(r-m)!},
\\&=& \dfrac{e^{-\lambda_n}F^{m-i}(x)(1-F(x))^i}{i!} (\lambda_n F (x))^{i-m} \sum_{r=m}^\infty \dfrac {r!}{(r-i)!} \dfrac{( \lambda_n F (x))^{r-i}}{(r-m)!},
\\&=& \dfrac{e^{-\lambda_n}F^{m-i}(x)(1-F(x))^i}{i!} (\lambda_n F (x))^{i-m} \sum_{r=m}^\infty \dfrac 1{(r-m)!}D^i (\lambda_n F (x))^r,
\\&=& \dfrac{e^{-\lambda_n}F^{m-i}(x)(1-F(x))^i}{i!} (\lambda_n F (x))^{i-m} D^i\left(\left( \lambda_n F (x) \right)^m e^{\lambda_nF(x)}\right),
\end{eqnarray*}
\begin{eqnarray*}
&=& \dfrac{e^{-\lambda_n}F^{m-i}(x)(1-F(x))^i}{i!} (\lambda_n F (x))^{i-m} e^{\lambda_n F (x) }\sum_{l=0}^i \binom il \dfrac {m!}{(m-l)!} ({\lambda_n F (x) })^{m-l},\\&=& \dfrac{e^{-\lambda_n(1-F(x)}F^{m-i}(x)(1-F(x))^i}{i!} \sum_{l=0}^i \binom il \dfrac {m!}{(m-l)!} (\lambda_n F (x))^{i-l},
\\&=& \dfrac{e^{-\lambda_n(1-F(x)}F^{m-i}(x)(n(1-F(x)))^i}{i!} \left({\dfrac {\lambda_n}n F (x) }\right)^i + \\&& \dfrac{e^{-\lambda_n(1-F(x)}F^{m-i}(x)(n(1-F(x)))^i}{i!} \sum_{l=1}^i \binom il \dfrac {m!}{(m-l)!} \dfrac {(\lambda_n F (x))^{i-l}}{n^i}.
\end{eqnarray*}
Substituting all these expressions in equation (\ref{eqn:gr-max1}), we get 
\begin{eqnarray*}
F_{k:N_n}(x)&=& e^{-\lambda_n(1-F(x)} F^m(x) + \sum_{i=1}^{k-1}  \dfrac{e^{-\lambda_n(1-F(x)}F^{m-i}(x)(n(1-F(x)))^i}{i!} \left(\dfrac {\lambda_n}n F (x) \right)^i +\\&& \sum_{i=1}^{k-1} \dfrac{e^{-\lambda_n(1-F(x)}F^{m-i}(x)(n(1-F(x)))^i}{i!} \sum_{l=1}^i \binom il \frac {m!}{(m-l)!} \dfrac {(\lambda_n F (x))^{i-l}}{n^i}.
\end{eqnarray*}
Replacing $\;x\;$ by $\;a_n x + b_n\;$ above and using the facts and $\;\lim\limits_{n \rightarrow \infty} e^{-\lambda_n(1-F(a_nx+b_n)} = G(x),\;$  we get 
\begin{eqnarray*}
F_{k}(x)&=& G(x)+\sum_{i=1}^{k-1} \frac {G(x) \set{-\ln G(x)}^i}{i!} (1) + \sum_{i=1}^{k-1} \frac {G(x) \set{-\ln G(x)}^i}{i!} \sum_{l=1}^i (0)
\\ &=& G(x)\sum_{r=0}^{k-1} \left(\dfrac {(-\ln G(x))^r}{r!}\right), \;\;\text{completing the proof of (b). }
\end{eqnarray*} 
\item[(c)] Let $\;D^r=\dfrac{d^r}{d(F(x))^r},\, r=1,2,3,\ldots.\;$ Considering the term corresponding to $\;i=0\;$ in (\ref{eqn:gr-max1}), we have 
\[ \sum_{r=m}^\infty F^r(x)\frac {e^{-\lambda_n} \lambda_n^{r-m}}{(r-m)!} = e^{-\lambda_n} F^m(x) \sum_{r=m}^\infty \frac{(\lambda_n F (x))^{r-m}}{(r-m)!} =e^{-\lambda_n(1-F(x)} F^m(x). \]
For $\;1 \leq i \leq k-1, \;$ in (\ref{eqn:gr-max1}), we now have 
\begin{eqnarray*}
 \sum_{r=n+1}^\infty F^r(x) \Set{\dfrac 1{-\ln (1-\theta_n)}} \dfrac {\theta_n^{r-n}}{r-n} 
 &=&  \Set{\dfrac 1{-\ln (1-\theta_n)}} F^n(x) \sum_{r=n+1}^\infty \frac {(\theta_nF(x))^{r-n}} {r-n}  \\&=&  \Set{\dfrac 1{-\ln (1-\theta_n)}} F^n(x) \Set{-\ln (1-\theta_nF(x))}
\end{eqnarray*}
By Leibnitz rule we have 
\begin{eqnarray*} && D^i  \set{ - F^n(x)  \ln(1-\theta_nF(x))} 
\\&=& -\frac {n!}{(n-i)!} F^{n-i}(x) \ln(1-\theta_nF(x)) -\sum_{l=0}^{i-1} \binom il \frac {n!}{(n-l)!} F^{n-l}(x) \frac {(-1)^{i-l-1}(i-l-1)!(-\theta_n)^{i-l}}{(1-\theta_nF(x))^{i-l}}
\end{eqnarray*}
Now consider the general $i$th term for $i\ge 1.$
\begin{eqnarray*}
&& \sum_{r=n+1}^\infty \binom ri  F^{r-i}(x) \set{1-F(x)}^i \Set{\dfrac 1{-\ln (1-\theta_n)}} \dfrac {\theta_n^{r-n}}{r-n} 
\\&=& \set{1-F(x)}^i \Set{\dfrac 1{-\ln (1-\theta_n)}} \sum_{r=n+1}^\infty \frac {r!}{i!(r-i)!} F^{r-i}(x) \frac {\theta_n^{r-n}}{r-n} 
\end{eqnarray*}
\begin{eqnarray*}
\\&=& \frac {\set{1-F(x)}^i}{i!} \Set{\dfrac 1{-\ln (1-\theta_n)}} D^i  \set{ - F^n(x)  \ln(1-\theta_nF(x))}
\\&=&   -\frac {\set{1-F(x)}^i}{i!} \Set{\dfrac 1{-\ln (1-\theta_n)}}  \frac {n!}{(n-i)!} F^{n-i}(x) \ln(1-\theta_nF(x)) \\&& -\frac {\set{1-F(x)}^i}{i!} \Set{\dfrac 1{-\ln (1-\theta_n)}} \sum_{l=0}^{i-1} \binom il \frac {n!}{(n-l)!} F^{n-l}(x) \frac {(-1)^{i-l-1}(i-l-1)!(-\theta_n)^{i-l}}{(1-\theta_nF(x))^{i-l}} 
\end{eqnarray*}
Substituting all these expressions in equation (\ref{eqn:gr-max1}), we get 
\begin{eqnarray*}
F_{k:N_n}&=&  \Set{\dfrac {-\ln (1-\theta_nF(x))}{-F(x) \ln (1-\theta_n)}} F^{n+1}(x) + \sum_{i=1}^{k-1}\frac {\set{1-F(x)}^i}{i!} \Set{\dfrac {-\ln(1-\theta_nF(x)}{-F(x) \ln (1-\theta_n)}}  \frac {n!}{(n-i)!} F^{n-i+1}(x)  -\\&& \sum_{i=1}^{k-1} \frac {\set{1-F(x)}^i}{i!} \Set{\dfrac 1{-\ln (1-\theta_n)}} \sum_{l=0}^{i-1} \binom il \frac {n!}{(n-l)!} F^{n-l}(x) \frac {(-1)^{i-l-1}(i-l-1)!(-\theta_n)^{i-l}}{(1-\theta_nF(x))^{i-l}} 
\end{eqnarray*}
Since  $\lim_{n\rightarrow \infty } \Set{\dfrac {-\ln (1-\theta_nF(a_nx+b_n))}{-F(a_nx+b_n) \ln (1-\theta_n)}}=1,$  we get 
\[ F_k(x) =  \lim_{n\rightarrow\infty} F_{k:N_n}( a_n x + b_n ) = G(x)+\sum_{i=1}^{k-1} \frac {(-\ln G(x))^i}{i!} -\sum_{i=1}^{k-1} 0=G(x)+\sum_{i=1}^{k-1} \frac {(-\ln G(x))^i}{i!}\]
This completes the proof of (c). 
\end{itemize}
\end{proof}

\begin{proof} [\textbf{Proof of Theorem \ref{Thm_FkG}:}] 
 Let $\;D^r=\dfrac{d^r}{d(q_n F(x))^r},\, r=1,2,3,\ldots.\;$ Considering the term corresponding to $\;i=0\;$ in (\ref{eqn:gr-max1}), we have 
\[ \sum_{r=m}^\infty F^r(x) p_nq_n^{r-m}=p_n F^m(x) \sum_{r=m}^\infty  (q_n F (x))^{r-m}= \left( \frac{p_nF^m(x)}{1-q_n F(x)}\right).\]
Applying (a) of Lemma \ref{lem:lr}, we get  $\;D^i\left(\frac {(q_nF(x))^m}{1-q_n F(x)}\right) = i!   \sum_{l=0}^i  \binom ml  \frac {(q_nF(x))^{m-l}}{(1-q_nF(x))^{i+1-l}}.$
For $\;1 \leq i \leq k-1, \;$ in (\ref{eqn:gr-max1}), we now have $\sum_{r=m}^\infty \binom ri  F^{r-i}(x)  (1-F(x))^i  p_n q_n^{r-m}$ 
\begin{eqnarray*}
&=& \frac{p_n F^{m-i}(x)(1-F(x))^i}{i!} (q_n F (x))^{i-m}\sum_{r=m}^\infty  \frac {r!}{(r-i)!} (q_n F (x))^{r-i},
\\&=& \frac{p_n F^{m-i}(x)(1-F(x))^i}{i!} ( q_n F (x))^{i-m}\sum_{r=m}^\infty D^i ( q_n F (x))^r,
\\ &=& \frac{p_n F^{m-i}(x)(1-F(x))^i}{i!} (q_n F (x))^{i-m} D^i \sum_{r=m}^\infty  (q_n F (x))^r,
\\&=& \frac{p_n F^{m-i}(x)(1-F(x))^i}{i!} (q_n F (x))^{i-m} D^i \left( \frac{(q_n F (x))^m}{1-q_nF(x)}\right),
\\&=& \frac {p_n F^{m-i}(x)(1-F(x))^i}{i!} (q_n F (x))^{i-m} i!   \sum_{l=0}^i  \binom ml  \frac {(q_nF(x))^{m-l}}{(1-q_nF(x))^{i+1-l}} \;\mbox{ by Lemma \ref{lem:lr},  }
\\&=& p_n F^{m-i}(x)(1-F(x))^i  \sum_{l=0}^i \binom ml \frac { (q_nF(x))^{i-l}}{(1-q_nF(x))^{i+1-l}},
\\&=& \frac{p_nq_n^i F^m(x)(1-F(x))^i}{(1-q_nF(x))^{i+1}} +\sum_{l=1}^i \binom ml \frac {p_n q_n^{i-l}F^{m-l}(x)(1-F(x))^i}{(1-q_nF(x))^{i+1-l}}.
\end{eqnarray*}
Substituting all these expressions in equation (\ref{eqn:gr-max1}), we get 
\begin{eqnarray*}
F_{k:N_n}(x) &=& P(X_{k:N_n} \le x ) \\&=&
\frac {p_n F^m(x) }{1-q_n F(x)}+ \sum_{i=1}^{k-1}\left(\frac{p_nq_n^i F^m(x)(1-F(x))^i}{(1-q_nF(x))^{i+1}} +\sum_{l=1}^i \binom ml \frac {p_n q_n^{i-l}F^{m-l}(x)(1-F(x))^i}{(1-q_nF(x))^{i+1-l}}\right).
\end{eqnarray*}
Replacing $\;x\;$ by $\;a_n x + b_n\;$ above and using the facts 
\begin{eqnarray*} \lim_{n \rightarrow \infty} np_n& = &1, \;\; \lim_{n \rightarrow \infty} q_n^k=  \lim_{n \rightarrow \infty} \left(1-\dfrac {np_n}n\right)^k  =1, \\   \lim_{n \rightarrow \infty} n(1-q_nF(a_n x + b_n)) & = &\lim_{n \rightarrow \infty}  n(p_nF(a_n x + b_n)+1-F(a_n x + b_n))= 1-\ln G(x), \end{eqnarray*} we get 
\begin{eqnarray*}
L_k(x)&=& \lim_{n\rightarrow\infty} \frac {np_n F^m(a_n x + b_n) }{n(1-q_n F(a_n x + b_n))}+\lim_{n\rightarrow\infty} \sum_{i=1}^{k-1} \frac{np_nq_n^i F^m(a_n x + b_n)(n(1-F(a_n x + b_n)))^i}{n^{i+1}(1-q_nF(a_n x + b_n))^{i+1}}+\\&&\lim_{n\rightarrow\infty} \sum_{i=1}^{k-1} \sum_{l=1}^i \binom ml \frac {np_n q_n^{i-l}F^{m-l}(a_n x + b_n)(n(1-F(a_n x + b_n)))^i}{n^{i+1}(1-q_nF(x))^{i+1-l}}, 
\\&=& \frac 1{1-\ln G(x)} + \sum_{i=1}^{k-1} \frac {(-\ln G(x))^i}{(1-\ln G(x))^{i+1}}+ \sum_{i=1}^{k-1} 0 
=  \frac 1{1-\ln G(x)} \sum_{i=1}^k \left( \frac {-\ln G(x)}{1-\ln G(x)}\right)^{i-1}, 
\end{eqnarray*}
$\;=1-\left( \frac {-\ln G(x)}{1-\ln G(x)}\right)^k,\;$ completing the proof.
\end{proof}

\begin{proof} [\textbf{Proof of Theorem \ref{Thm-FkG2}:}]
For $\;-\infty < x<y < \infty,\;$ we have $\;F(x) \le F(y),\;$ so that  
\begin{eqnarray*}
&&1-\ln F(x) \ge 1-\ln F(y) \Rightarrow 1-\dfrac 1{1-\ln F(x)} \ge 1-\dfrac 1{1-\ln F(y)} \\ &\Rightarrow&
 \left(\frac {-\ln F(x)} {1-\ln F(x)}\right)^k \ge \left(\frac {-\ln F(y)} {1-\ln F(y)}\right)^k \Rightarrow R_k(x) \le R_k(y),
\end{eqnarray*}
showing that $\;R_k\;$ is non-decreasing. Since $\;F\;$ is right continuous, we have 
\[ \lim_{\epsilon \rightarrow 0} R_k(x+\epsilon)=\lim_{\epsilon \rightarrow 0} \left(1-\left(\dfrac {-\ln F(x+\epsilon)} {1-\ln F(x+\epsilon)}\right)^k\right)=R_k(x),\] so that $\;R_k\;$ is right continuous. Further $\;R_k(+\infty)=1\;$ and $\;R_k(-\infty)=0,\;$ and hence $\;R_k\;$ is a df. And, if $\;F\;$ is absolutely continuous with pdf $\;f,\;$ then $\;R_k\;$ is absolutely continuous with pdf $\;r_k(x)=\dfrac {kf(x) (-\ln F(x))^{k-1}}{F(x) (1-\ln F(x))^{k+1}}.\;$ 
Note that $\; \lim_{x\rightarrow r(F)} R_k(x)=1,\;$ and hence $\;r(R_k)\le r(F) \le \infty.\;$ If possible let $\;r(R_k)<r(F).\;$ Then 
\begin{eqnarray*} & & F(r(R_k)) < 1 \Rightarrow 1 -\ln F(r(R_k)) >1 \Rightarrow  \dfrac 1{1 -\ln F(r(R_k))} <1 \Rightarrow  1-\dfrac 1{1 -\ln F(r(R_k))} >0
\\&\Rightarrow&  \dfrac {-\ln F(r(R_k))}{1 -\ln F(r(R_k))} >0 \Rightarrow 1-\left( \dfrac {-\ln F(r(R_k))}{1 -\ln F(r(R_k))}\right)^k <1 \Rightarrow R_k(r(R_k))<1,
\end{eqnarray*} 
which is a contradiction, proving $\;r(R_k)=r(F).$ 
We have  \[ \lim_{x\rightarrow \infty} \frac {1-R_k(x)}{1-H_k(x)}= \lim_{x\rightarrow \infty} \left( \dfrac {-\ln F(x)}{1-F(x)}\right)^k \dfrac 1{(1-\ln F(x))^k} = 1,\] so that $\;1-R_k \sim 1-H_k,\;$ or $\;R_k\;$ is tail equivalent to $\;H_k,\;$ proving (a).
Because of the tail equivalence of $\;R_k\;$ and $\;H_k,\;$ (b), (c) and (d) follow from the corresponding results in Theorem \ref{Thm-Fk}.
\end{proof}

\begin{proof} [\textbf{Proof of Theorem \ref{Thm-FkG3}:}]
With $\;W(x)=R_k(x),\;$ we have 
\[ \dfrac {dR_k(x)}{dx}=r_k(x)=\dfrac {kf(x) (-\ln F(x))^{k-1}}{F(x) (1-\ln F(x))^{k+1}}=(1-R_k(x)) h_1(x),\] where $\;h_1(x)=\dfrac {-k f(x)}{F(x) (1-\ln F(x)) \ln F(x)}>0,$ so that (\ref{eqn:burr}) is satisfied and hence $\;R_k\;$ belongs to the Burr family.
\end{proof}

\begin{proof} [\textbf{Proof of Theorem \ref{Thm-FkNB1}:}]
Let $\;D^r=\dfrac{d^r}{d(q_n F(x))^r},\, r=1,2,3,\ldots.\;$ Note that 
\[ \sum_{l=m}^\infty \binom {l-m+r-1}{l-m} q_n^{l-m}=(1-q_n)^{-r}   \text{  and  } \sum_{l=m}^\infty \binom {l-m+r-1}{l-m} q_n^l=q_n^m (1-q_n)^{-r}.\]
Considering the term corresponding to $\;i=0\;$ in (\ref{eqn:gr-max1}), we have 
\[ \sum_{l=m}^\infty  F^l(x)  \binom {l-m+r-1}{l-m} p_n^r q_n^{l-m}= p_n^r F^m(x) \sum_{l=m}^\infty \binom {l-m+r-1}{l-m} (q_n F (x))^{l-m}=\dfrac {p_n^rF^m(x)} {(1-q_n F(x))^r}.\]
For $\;1 \leq i \leq k-1, \;$ in (\ref{eqn:gr-max1}),  we now have 
\begin{eqnarray*}
& &\sum_{l=m}^\infty \binom li  F^{l-i}(x) (1-F(x))^i \binom {l-m+r-1}{l-m} p_n^r q_n^{l-m}
\\&=& \dfrac{p_n^r F^{m-i}(x)(1-F(x))^i}{i!} (q_n F (x))^{i-m}\sum_{l=m}^\infty \binom {l-m+r-1}{l-m} \dfrac {l!} {(l-i)!} (q_n F (x))^{l-i},
\\&=& \dfrac{p_n^r F^{m-i}(x)(1-F(x))^i}{i!} (q_n F (x))^{i-m}\sum_{l=m}^\infty \binom {l-m+r-1}{l-m} D^i (q_n F (x))^l,
\\&=& \dfrac{p_n^r F^{m-i}(x)(1-F(x))^i}{i!} (q_n F (x))^{i-m} D^i\left(\dfrac {(q_nF(x))^m}{(1-q_n F(x))^r}\right),
\\&=& \dfrac{p_n^r F^{m-i}(x)(1-F(x))^i}{i!} i! (q_n F (x))^{i-m}\sum_{l=0}^i \binom m{i-l} \binom {r-1+l}l \dfrac {(q_nF(x))^{m-i+l}}{(1-q_nF(x))^{r+l}},
\\&=& p_n^r (1-F(x))^i \sum_{l=0}^i \binom m{i-l} \binom {r-1+l}l \dfrac {(q_nF(x))^l F^{m-i}(x)}{(1-q_nF(x))^{r+l}}, 
\\&=& p_n^r (1-F(x))^i \sum_{l=0}^i \binom m{i-l} \binom {r-1+l}l \dfrac {q_n^l F^{m-i+l}(x)}{(1-q_nF(x))^{r+l}}, 
\\&=& p_n^r (1-F(x))^i \sum_{l=0}^{i-1} \binom m{i-l} \binom {r-1+l}l \dfrac {q_n^l F^{m-i+l}(x)}{(1-q_nF(x))^{r+l}} + \\&& (np_n)^r (n(1-F(x)))^i \binom {r+i-1}i\dfrac {q_n^iF^m(x)}{(n(1-q_nF(x)))^{r+i}}.
\end{eqnarray*}
Substituting we have 
\begin{eqnarray*}
F_{k:N_n}(x) &=&\dfrac {p_n^rF^m(x)} {(1-q_n F(x))^r}+\sum_{i=1}^{k-1} p_n^r (1-F(x))^i \sum_{l=0}^{i-1} \binom m{i-l} \binom {r-1+l}l \dfrac {q_n^l F^{m-i+l}(x)}{(1-q_nF(x))^{r+l}} + \\&&\sum_{i=1}^{k-1}  (np_n)^r (n(1-F(x)))^i \binom {r+i-1}i\dfrac {q_n^iF^m(x)}{(n(1-q_nF(x)))^{r+i}}. 
\end{eqnarray*}
Replacing $\;x\;$ by $\;a_n x + b_n\;$ above and using the facts 
$\;\lim_{n \rightarrow \infty} np_n =1, \;$  $\lim_{n \rightarrow \infty} F(a_n x + b_n) = 1,\;$ $\;\lim_{n \rightarrow \infty} n(1-F(a_n x + b_n))=  -\ln G(x),\;$ $\;\lim_{n\rightarrow \infty} n(1-q_nF(a_n x + b_n))= 1-\ln G(x),\;$ and $\;\lim_{n\rightarrow \infty} q_n^k = \lim_{n\rightarrow \infty} (1-\dfrac {np_n}n)^k = 1,\;$ we get 
\begin{eqnarray*}
S_k(x)&=&  \dfrac 1{(1-\ln G(x))^{r}}+\sum_{i=1}^{k-1}  0 + \sum_{i=1}^{k-1}  \binom {r+i-1}i \dfrac {(-\ln G(x))^i}{(1-\ln G(x))^{r+i}},
\\&=& \sum_{i=0}^{k-1}  \binom {r+i-1}i \dfrac {(-\ln G(x))^i}{(1-\ln G(x))^{r+i}},
\end{eqnarray*}
completing the proof.
\end{proof}

\begin{proof} [\textbf{Proof of Theorem  \ref{Thm_FkNB2}:}]
The proof for the relation (\ref{eqn:recdf1}) is evident from the definition of $\;T_k\;$ since  $\; T_{k+1}(x)- T_k(x)=\binom {k+r-1}k  \dfrac {(-\ln F(x))^k}{(1-\ln F(x))^{k+r}}.$ The proof of (\ref{eqn:recpdf1}) is by induction.\\ For $\;k=1,\;$ the df $\;T_2(x)=\dfrac 1{(1-\ln F(x))^r}+\binom r1 \dfrac {(-\ln F(x))}{(1-\ln F(x))^{r+1}}. \;$ Hence the pdf becomes 
\begin{eqnarray*}
t_2(x)&=&\dfrac {r f(x)}{F(x)} \dfrac 1{(1-\ln F(x))^{r+1}} +\binom r1  (-\frac {f(x)}{F(x)}) \left(\dfrac 1{(1-\ln F(x))^{r+1}}-\dfrac {(r+1)(-\ln F(x))}{(1-\ln F(x))^{r+2}}\right),
 \\&=&\binom r1  \left(\dfrac {f(x)}{F(x)}\right)\dfrac {(r+1)(-\ln F(x))}{(1-\ln F(x))^{r+2}} =\dfrac 1{B(r,2)} \dfrac {f(x)}{F(x)}\dfrac {(-\ln F(x))}{(1-\ln F(x))^{r+2}}
 \end{eqnarray*}
which agrees with (\ref{eqn:recpdf1}) for $\;k=1.\;$ 
Assuming that the relation (\ref{eqn:recpdf1}) is true for $\;k=s,\;$ 
we have $\;T_{s+1}(x)=T_s(x)+\binom {s+r-1}s  \dfrac {(-\ln F(x))^s}{(1-\ln F(x))^{s+r}}\;$  with  $\;t_{s+1}(x)=\dfrac 1{B(r,s+1)} \dfrac {f(x)}{F(x)}  \dfrac {(-\ln F(x))^s}{(1-\ln F(x))^{r+s+1}}.\;$ Now consider $\;T_{s+2}(x)=T_{s+1}(x)+\binom {s+r}{s+1}  \dfrac {(-\ln F(x))^{s+1}}{(1-\ln F(x))^{s+1+r}}.$\\
Differentiating with respect to $\;x\;$ we get 
\begin{eqnarray*}
&& t_{s+2}(x)-t_{s+1}(x)= \binom {s+r}{s+1} \left(-\dfrac {f(x)}{F(x)}\right) \left(\dfrac {(s+1)(-\ln F(x))^s}{(1-\ln F(x))^{s+1+r}}-\dfrac {(s+r+1)(-\ln F(x))^{s+1}}{(1-\ln F(x))^{s+r+2}}\right),
\\&=& - (s+1) \binom {s+r}{s+1} \dfrac {f(x)}{F(x)}\dfrac {(-\ln F(x))^s}{(1-\ln F(x))^{s+1+r}} +(s+2)\dfrac {s+r+1}{s+2}\binom {s+r}{s+1} \dfrac {f(x)}{F(x)} \dfrac {(-\ln F(x))^{s+1}}{(1-\ln F(x))^{s+r+2}},
\\&=& -\dfrac 1{B(r,s+1)}\dfrac {f(x)}{F(x)}\dfrac {(-\ln F(x))^s}{(1-\ln F(x))^{s+1+r}}+ (s+2)\binom {s+r+1}{s+2} \dfrac {f(x)}{F(x)} \dfrac {(-\ln F(x))^{s+1}}{(1-\ln F(x))^{s+r+2}},
\\&=&\dfrac 1{B(r,s+2)}\dfrac {f(x)}{F(x)} \dfrac {(-\ln F(x))^{s+1}}{(1-\ln F(x))^{s+r+2}}.
\end{eqnarray*}
Thus the relation (\ref{eqn:recpdf1}) follows. Note that $\;t_{k+1}(x)>0.\;$ Hence the proof.
\end{proof}

\begin{proof} [\textbf{Proof of Theorem \ref{Thm_FkNB3}:}]
\begin{enumerate}
\item[(a)] By Theorem \ref{Thm_FkNB2}, $\;T_k\;$ is a df with pdf  $\;t_k(x)= $ \\ $ \dfrac 1{B(r,k)} \dfrac {f(x)}{F(x)}  \dfrac {(-\ln F(x))^{k-1}}{(1-\ln F(x))^{r+k}}.\;$ Note that $\;\lim_{x\rightarrow r(F)} T_k(x)=$ \\
$\lim_{x\rightarrow r(F)}\left( \dfrac 1{(1-\ln F(x))^r} + \sum_{l=1}^{k-1} \binom {l+r-1}l  \dfrac {(-\ln F(x))^l}{(1-\ln F(x))^{l+r}}	\right)=1,\;$ so that $\;r(T_k) \le r(F) \le \infty.\;$  If possible, let $\;r(T_k)<r(F).\;$ 
Since $\;0<F(r(T_k))<F(r(F))=1,\;$ it follows that $\;0<\dfrac 1{1-\ln F(r(T_k))}<1\;$ and $\;0<\dfrac {-\ln F(r(T_k))} {1-\ln F(r(T_k))}<1.\;$ Then 
\[T_k(r(T_k))=\sum_{l=0}^{k-1}  \binom {l+r-1}l  \dfrac {(-\ln F(r(T_k)))^l}{(1-\ln F(r(T_k)))^{l+r}} <\sum_{l=0}^\infty \binom {l+r-1}l  \dfrac {(-\ln F(r(T_k)))^l}{(1-\ln F(r(T_k)))^{l+r}}=1,\]
which is a contradiction, proving that $\;r(T_k)=r(F).$ We have  
\[ \lim_{x\rightarrow \infty} \dfrac {1-T_k(x)}{1-H_k(x)}=\lim_{x\rightarrow \infty} \dfrac {-t_k(x)}{-k(1-F(x))^{k-1}f(x)}=\lim_{x\rightarrow \infty}  \dfrac 1{B(r,k)} \dfrac 1{kF(x)} \dfrac {u^{k-1}(x)}{(1-\ln F(x))^{r+k}}= \dfrac 1{kB(r,k)}.\]
Thus $\;1-T_k \sim 1-H_k,\;$ or $\;T_k\;$ is tail equivalent to $\;H_k.\;$
Because of the tail equivalence of $\;T_k\;$ and $\;H_k,\;$ (b), (c) and (d) follow from the corresponding results in Theorem \ref{Thm-Fk}. 
\end{enumerate}
\end{proof}

\begin{proof} [\textbf{Proof of Theorem \ref{Thm_FkBN}:}]
\begin{enumerate}
\item[(i)]  Clearly $B_k$ is right continuous with $\;B_k(+\infty)=1\;$ and $\;B_k(-\infty)=0.$ We have $\;B_k(r(F))= E_{\tau}(F^\tau(r(F))) + \sum_{i=1}^{k-1}\frac{(-\log F(r(F)))^i}{i!}E_{\tau}(\tau^i F^\tau(r(F)))=$ $E_\tau(1)+\sum_{i=1}^{k-1}0=1.\;$ Hence $r(B_k)\le r(F).$ If possible, let $r(B_k) < r(F).$ Then $F(r(B_k))<1$ and 
\begin{eqnarray*}
 B_k(r(B_k))&=&  \sum_{i=0}^{k-1}\frac{(-\log F(r(B_k)))^i}{i!} E_{\tau}(\tau^i F^\tau(r(B_k))) \\ 
 &=&  E_\tau\set{ \sum_{i=0}^{k-1}\dfrac{(-\tau \log F(r(B_k)))^i}{i!} F^\tau(r(B_k))} <E_\tau\set{ \sum_{i=0}^\infty \dfrac{(-\tau \log F(r(B_k)))^i}{i!} F^\tau(r(B_k))}
\\&=&  E_\tau \set{ \exp {(-\tau \log F(r(B_k)))} F^y(r(B_k))} = E_{\tau}(\tau^0) =1,
\end{eqnarray*} a contradiction proving that $\;r(B_k) = r(F).\;$ The recurrence for $B_{k+1}$ follows trivially by definition of $B_k.$ 
The pdf $b_{k+1}(x)$ is proved by induction on $k.$ Considering the case of $k=1,2,$ we have 
$B_1(x)=$ $  E_{\tau}(F^\tau(x)),\;$ $B_2(x)=$ $V_1(x)+\frac {(-\log F(x))^1}{1!}  E_{\tau}(\tau F^\tau(x)).\;$ Differentiating with respect to $x$ and taking derivative  under the expectation sign 
\begin{eqnarray*}
b_1(x)&=& E_{\tau}(\tau F^{\tau-1} (x)) f(x),
\\b_2(x)&=&v_1(x)+\frac {-f(x)}{F(x)} E_{\tau}(\tau F^\tau (x)) +\frac {(-\log F(x))^1}{1!}  E_{\tau}(\tau^2 F^{\tau-1} (x)) f(x),
\\&=& \frac {(-\log F(x))^1}{1!}  E_{\tau}(\tau^2 F^{\tau-1} (x)) f(x).
\end{eqnarray*}
Assuming the result to be true for $k,$ we have $b_k(x) = \dfrac{(-\log F(x))^{k-1}}{(k-1)!}E_{\tau}(\tau^k F^{\tau-1} (x)) f(x)$ and 
\begin{eqnarray*}
B_{k+1}(x)&=& B_k(x)+\dfrac{(-\log F(x))^k}{k!} E_{\tau}(\tau^k F^\tau(x)),\\
b_{k+1}(x)&=&b_k(x)+\frac {-(-\log F(x))^{k-1}f(x)}{(k-1)! F(x)} E_{\tau}(\tau^k F^\tau(x)) + \frac {(-\log F(x))^k}{k!} E_{\tau}(\tau^{k+1} F^{\tau-1} (x))f(x),
\\&=&  \frac {(-\log F(x))^k}{k!} E_{\tau}(\tau^{k+1} F^{\tau-1} (x))f(x).
\end{eqnarray*}
Hence the result. 
\item[(ii)] We have 
\begin{eqnarray*}
\lim_{x\to  \infty} \frac {1-B_k(x)}{(1-F(x))^k}&=&\lim_{x\to  \infty} \frac {-b_k(x)}{-kf(x) (1-F(x))^{k-1}},
\\&=& \lim_{x\to  \infty} \dfrac{(-\log F(x))^{k-1}}{(k-1)!k (1-F(x))^{k-1}} E_{\tau}(\tau^k F^{\tau-1} (x))f(x),
\\&=& \lim_{x\to  \infty} \dfrac{u^{k-1}(x)}{k!} E_{\tau}(\tau^k F^{\tau-1} (x))f(x)= \frac 1{k!} E_\tau(\tau^k \times 1) = \frac 1{k!}E_\tau(\tau^k).
\end{eqnarray*}
\end{enumerate}
\end{proof}

\begin{proof} [\textbf{Proof of Theorem \ref{Thm_StOr}:}]
 \begin{enumerate}

\item[(i)]  Proved in theorem (\ref{Thm-FkW}). 
\item[(ii)] Since $0\le F(x)\le 1$ it also follows that $0\le 1-F(x)\le 1.$ The results follow by observing that $\;\frac {1- H_{k+1}(x)}{1-H_k(x)}<1\;$, $\;\frac {1- R_{k+1}(x)}{1-R_k(x)}<1\;$, $\;F_{k+1}(x)-F_k(x)\ge0\;$ and $\;T_{k+1}(x)-T_k(x)\ge0.$
\item[(iii)] By recurrence relation we have $\displaystyle U_{k+1}(x)-U_k(x) = U_1(x) - F(x) \sum_{l=1}^k \frac {(-\ln F(x))^{l-1}}{l!}.$
Observe that 
\begin{eqnarray*}
&& F(x) \sum_{l=1}^{k-1} \frac {(-\ln F(x))^{l-1}}{l!} = \frac {F(x)}{-\ln F(x)} \sum_{l=1}^k \frac {(-\ln F(x))^l}{l!} < \frac {F(x)}{-\ln F(x)} \sum_{l=1}^\infty \frac {(-\ln F(x))^l}{l!} \\&=&  \frac {F(x)}{-\ln F(x)} \Set{e^{-ln F(x)}-1}=\frac {F(x)}{-\ln F(x)} \Set{\frac 1{ F(x)}-1} = \frac {1-F(x)}{-\ln F(x)}=U_1(x)
\end{eqnarray*}
Hence $\;U_{k+1}(x)-U_k(x) >0\;$ and result follows.\\ From (i) $U_1(x) \le_{\mbox{st}} F(x) \Rightarrow U_1(x) \ge F(x).$  Assume that $U_k(x) \ge F(x).$  Then by recurrence relation we have $U_{k+1}(x) > U_k(x) \ge F(x).$ 
\item[(iv)] From (i) $\;U_1(x)>F(x)\ \Rightarrow \frac 1{F(x)} -1 \ge -\ln F(x)$ it follows that 
 \begin{eqnarray*}
 && 1- \ln F(x) \le \frac 1{F(x)} \Rightarrow \frac 1{1- \ln F(x) } \ge F(x) \Rightarrow 1- \frac 1{1- \ln F(x) } \le 1-F(x) \\
 &\Rightarrow& \Set{\frac {-\ln F(x)}{1- \ln F(x)}}^k \le (1-F(x))^k \le 1-F(x) \Rightarrow 1-R_k(x) \le 1-H_k(x) \le 1-F(x)
 \end{eqnarray*}
 Hence $R_k(x) \le_{\mbox{st}} H_k(x) \le_{\mbox{st}} F(x).$
\item[(v)]  Since $\;1-\ln F(x) < \frac 1{F(x)}\;$
\begin{eqnarray*}
1-F_k(x)&=&1-F(x)\sum_{r=0}^{k-1}\Set{\frac {(-\ln F(x))^r}{r!}}=(1-F(x))-F(x)\sum_{r=1}^{k-1}\Set{\frac {(-\ln F(x))^r}{r!}}<1-F(x)
\\1-T_k(x)&=&1-\sum_{l=0}^{k-1} \binom {l+r-1}l  \frac {\set{-\ln F(x)}^l}{(1-\ln F(x))^{r+l}} 
\\&=& 1-\frac 1{(1-\ln F(x))^r}- \sum_{l=0}^{k-1} \binom {l+r-1}l  \frac {\set{-\ln F(x)}^l}{(1-\ln F(x))^{r+l}} < 1-\frac 1{(1-\ln F(x))^r} < 1-F^r(x)
\end{eqnarray*}
\item[(vi)] For $0<p<1$ and $r\ge 1$ such that $p+q=1,$  setting $p=\dfrac 1{1-\ln F(x)}$ it follows that $R_k(x)=1-q^k$ and $T_k(x)=\sum\limits_{l=0}^{k-1} \binom {l+r-1}l p^r q^l.$\\ Let $X\sim R_k(x)=1-q^k$ and $Y\sim T_k(x)=\sum\limits_{l=0}^{k-1} \binom {l+r-1}l p^r q^l.$\\ 
Since the sum of iid geometric rvs is a negative binomial rv it follows that $\set{X>x} \subset \set{Y>x}$ and hence $P(X>x) \le P(Y>x) \Rightarrow  1-R_k(x) \le 1-T_k(x)$
\item[(vii)] Setting $\xi=-\ln F(x)$ one gets $\;U_k(x)\;$
\begin{eqnarray*}
&=& k\Set{\frac {1-F(x)}{-\ln F(x)}} -F(x) \sum_{l=1}^{k-1} (k-l) \frac {(-\ln F(x))^{l-1}}{l!} 
= k\Set{\frac {1-e^{-\xi}}{\xi}}-\sum_{l=1}^{k-1} (k-l) \frac {e^{-\xi} \xi^{l-1}}{l!} 
\\&=& k\sum_{j=1}^\infty \frac{e^{-\xi} \xi^{j-1}}{j!}-\sum_{l=1}^{k-1} (k-l) \frac {e^{-\xi} \xi^{l-1}}{l!} = k\sum_{j=k}^\infty \frac{e^{-\xi} \xi^{j-1}}{j!}+\sum_{l=1}^{k-1} \frac {e^{-\xi} \xi^{l-1}}{(l-1)!} 
\end{eqnarray*}
\begin{eqnarray*}
=k\sum_{j=k+1}^\infty \frac{e^{-\xi} \xi^{j-1}}{j!}+\sum_{l=1}^k \frac {e^{-\xi} \xi^{l-1}}{(l-1)!} >\sum_{l=1}^k \frac {e^{-\xi} \xi^{l-1}}{(l-1)!} =\sum_{l=0}^{k-1} \frac {F(x) \set{-\ln F((x)}^l}{l!} =F_k(x)
\end{eqnarray*}
%\item[(viii)] The result seems to be true because of the following figures.
%\begin{figure}[htbp]
	%\centering
		%\includegraphics{E:/R_files/plot_2.jpeg}
	%\label{fig:plot_2}
%\end{figure}
%However we are unable to give proof and leave this as conjecture. 

 \end{enumerate}
\end{proof}

\section{Appendix}
\begin{thm} \label{T-l-max}
\begin{enumerate}
\item[(a)] $F \in {\mathcal D}_{l}(\Phi_{\alpha})\;$ for some $\;\alpha > 0\;$ iff
$\;1 - F\;$ is regularly varying with exponent $\;- \alpha,\;$ that is, $\;\lim_{t \rightarrow \infty} \dfrac{1 - F(tx)}{1 - F(t)} = x^{- \alpha}, \, x > 0.\;$ In this case, one may choose $\;a_n = F^{-}\left(1 - \dfrac{1}{n}\right)\;$ and $\;b_n = 0\;$ so that (\ref{Introduction_e1}) holds with $G = \Phi_{\alpha}.$ 
\item[(b)] $F \in {\mathcal D}_{l}(\Psi_{\alpha})\;$ for some $\;\alpha > 0\;$ iff
$\;r(F) < \infty\;$ and $\;1 - F\left( r(F) - \frac{1}{.}\right)\;$ is regularly varying with exponent $\;- \alpha.\;$ In this case, one may choose $\;a_n = F^{-}\left(r(F) - \dfrac{1}{n}\right)\;$ and $\;b_n = r(F)\;$ so that (\ref{Introduction_e1}) holds with $G = \Psi_{\alpha}.$ 
\item[(c)] $F \in {\mathcal D}_{l}(\Lambda)\;$ iff there exists a positive function
$\;v\;$ such that $\;\lim_{t \uparrow r(F)} \dfrac{1 - F\left(t + v(t)x\right)}{1 - F(t)} = e^{-x}, \, x \geq 0, \, r(F) \leq \infty.\;$ If this condition holds for some $\;v,\;$ then \\$\;\int_{a}^{r(F)}\left(1 - F(s)\right) ds < \infty, \, a < r(F),\;$ and the condition holds with the choice $\;v(t) = \dfrac{\int_{t}^{r(F)}\left(1 - F(s)\right) ds}{\left(1 -
F(t)\right)}\;$ and one may choose $\;a_n = v(b_n)\;$ and $\;b_n = F^{-}\left(1 - \dfrac{1}{n}\right) \;$ so that (\ref{Introduction_e1}) holds with $G = \Lambda.$  One may also choose $\;a_n = F^{-}\left(1 - \dfrac{1}{ne}\right) - b_n, \, b_n = F^{-}\left(1 -
\dfrac{1}{n}\right). \;$ This condition was by Worm (1998). Also, $\;v\;$ may be taken as the mean residual life time of a rv $\;X\;$ given $\;X > t\;$ where $\;X\;$ has df $\;F.\;$ 
\item[(d)] (Proposition 1.1(a), Resnick, 1987) A df $\;F\;$ is called a von-Mises function if there exists $\;z < r(F)\;$ such that for constant $\;c > 0,\;$ and $\;z < x < r(F),\;$ $\;1-F(x)=$ $c \exp\left( -\int_{z}^{x}\dfrac{1}{f(u)}du\right), \;$ where $\;f(u) > 0,$ $z<u<r(F),\;$ and $\;f\;$ is absolutely continuous on $\;(z,r(F))\;$ with density $\;f^{\prime}(u)\;$ and $\;\lim_{x\uparrow r(F)} f^{\prime}(u) = 0,\;$ and $\;f\;$ is called the auxiliary function. If df $\;F\;$ is a von-Mises function, then $\;F \in D_l(\Lambda)\;$ with norming constants $\;a_n = f(b_n),$ $b_n = F^-\left(1-\dfrac{1}{n}\right),\;$ so that (\ref{Introduction_e1}) holds with $G = \Lambda.$ 
\item[(e)] (Proposition 1.1(b), Resnick, 1987) Suppose df $\;F\;$ is absolutely continuous with $\;F^{\prime\prime}(x) < 0, x \in (z, r(F)),\;$ and if 
\begin{equation}  \lim_{x\uparrow r(F)} \dfrac{(1-F(x))F^{\prime\prime}(x)}{(F^{\prime}(x))^2} = -1, \label{von-Mises} \end{equation} then $\;F\;$ is a von-Mises function and $\;F \in D_l(\Lambda),\;$ and we can choose auxiliary function as $\;f = \dfrac{1-F}{F^{\prime}}.\;$ Conversely, a twice differentiable von-Mises function satisfies $\;(\ref{von-Mises}).$
\end{enumerate}
\end{thm}

Acknowledgement: The first author thanks Prof. R Vasudeva for bringing the thesis Vasantalakshmi M.S. (2010) and a particular case of the problem discussed to his notice.

\bibliographystyle{amsplain}

\begin{thebibliography}{}
\bibitem{1} Barakat, H.M. and Nigm, E.M. (2002). Extreme order statistics under power normalization and random sample size, {\it Kuwait Journal of Science and Engineering,} Vol. 29 (1), 27-41. 
\bibitem{2} Burr, I. W. (1942). Cumulative frequency distributions, {\it Annals of Mathematical Statistics,} Vol. 13, 215-232.
\bibitem{3} Galambos, J. (1987). {\it The Asymptotic Theory of Extreme Order Statistics,} 2nd edition, Krieger, Malabar.
\bibitem{4} Feller, W. (1971). {\it An Introduction to Probability Theory and Its Applications,} Vol. 2, 2nd edition, John Wiley and Sons.
\bibitem{5} Embrechts, P., Kluppelberg, C., and Mikosch, T. (1997). {\it Modelling Extremal Events for Insurance and Finance,} Springer-Verlag, Berlin Heidelberg.
\bibitem{6} Peng Z., Jiang Q., and Nadarajah S. (2012). Limiting distributions of extreme order statistics under power normalization and random index, {\it Stochastics,} Vol. 84 (4),  553-560. 
\bibitem{7}	Resnick, S. I. (1987). {\it Extreme Values, Regular Variation, and Point Processes,} Springer-Verlag.
\bibitem{8} Vasantalakshmi, M.S.,(2010) {\it Limit Theorems For Sums, Extremes and Related Stochastic Processes,} Unpublished Ph.D. thesis submitted to University of Mysore, Mysuru 570 006, India.
\bibitem{9} Worm 

\end{thebibliography}

\end{document}